\documentclass[reqno,11pt,psamsfonts]{amsart}
\usepackage[all]{xy}
\usepackage{amsthm}
\usepackage{amssymb}
\usepackage{amsmath,amscd}

\def\cal{\mathcal}

\def\AA{{\Bbb A}}

\def\NN{{\Bbb N}}
\def\PP{{\Bbb P}}

\def\ZZ{{\Bbb Z}}

\def\boxtimes{\setbox0\hbox{$\Box$}\copy0\kern-\wd0\hbox{$\times$}}

\newcommand{\lotimes}{\otimes^{\bf{L}}}

\def\diag{\operatorname {diag}}

\def\End{\operatorname {E nd}}
\def\Ext{\operatorname {Ext}}

\def\Hom{\operatorname {Hom}}
\def\id{\operatorname {id}}

\def\Ker{\operatorname {ker}}

\def\Spec{\operatorname {Spec}}

\def\Tor{\operatorname {Tor}}

\def\Aut{\operatorname{Aut}}

\def\fchar{\operatorname{char}}

\def\depth{\operatorname{depth}}

\def\det{\operatorname{det}}
\def\dim{\operatorname{dim}}

\def\End{\operatorname{End}}

\def\Ext{\operatorname{Ext}}

\def\uExt{\operatorname{\underline{Ext}}}

\def\gldim{\operatorname{gldim}}
\def\grmod{\operatorname{grmod}}

\def\GrMod{\operatorname{GrMod}}

\def\Hom{\operatorname{Hom}}
\def\RHom{\operatorname{RHom}}

\def\uHom{\operatorname{\underline{Hom}}}

\def\Id{\operatorname{Id}}
\def\Im{\operatorname{Im}}

\def\Ker{\operatorname{Ker}}

\def\lin{\operatorname{lin}}

\def\mod{\operatorname{mod}}
\def\Mod{\operatorname{Mod}}

\def\Proj{\operatorname{Proj}}
\def\proj{\operatorname{proj}}

\def\tails{\operatorname{tails}}
\def\Tails{\operatorname{Tails}}
\def\Tor{\operatorname{Tor}}

\def\tors{\operatorname{tors}}
\def\Tors{\operatorname{Tors}}

\def\coh{\operatorname{coh}}

\def\uEnd{\operatorname{\underline{End}}}
\def\uExt{\operatorname{\underline{Ext}}}

\def\uHom{\operatorname{\underline{Hom}}}

\let\oldtext\text
\def\text#1{\oldtext{\normalshape #1}}

\def\a{\alpha}
\def\b{\beta}

\def\e{\epsilon}

\def\cA{{\cal A}}

\def\cC{{\cal C}}
\def\cD{{\cal D}}

\def\cK{{\cal K}}

\def\cM{{\cal M}}
\def\cN{{\cal N}}
\def\cO{{\cal O}}

\def\cS{{\cal S}}
\def\cT{{\cal T}}

\def\add{\operatorname{add}}

\def\SL{\operatorname{SL}}

\def\<{\langle}
\def\>{\rangle}
\def\CM{\operatorname {CM}}
\def\uCM{\underline {\operatorname
{CM}}}
\def\GrAut{\operatorname{GrAut}}

\def\HSL{\operatorname{HSL}}

\def\ugrmod{\underline{\grmod}}

\def\thick{\operatorname{thick}}
\def\add{\operatorname{add}}

\newtheorem{lemma}{Lemma}[section]
\newtheorem{proposition}[lemma]{Proposition}
\newtheorem{theorem}[lemma]{Theorem}
\newtheorem{corollary}[lemma]{Corollary}

\newtheorem{prop}[lemma]{Proposition}

{

}

\theoremstyle{definition}

\newtheorem{example}[lemma]{Example}
\newtheorem{definition}[lemma]{Definition}

\newtheorem{setting}[lemma]{Setting}

\theoremstyle{remark}

\newtheorem{remark}[lemma]{Remark}

\usepackage{geometry}
\usepackage[all]{xy}
\usepackage{graphicx}

\def\grproj{\operatorname {grproj}}

\numberwithin{equation}{section}

\begin{document}
\pagenumbering{arabic}

\title[Stable Categories over Noncommutative Quotient Singularities]
{Stable Categories of Graded Maximal Cohen-Macaulay Modules over Noncommutative Quotient Singularities}

\author{Izuru Mori}

\address{
Department of Mathematics,
Faculty of Science,
Shizuoka University,
836 Ohya, Suruga-ku, Shizuoka 422-8529, JAPAN}

\email{mori.izuru@shizuoka.ac.jp}

\author{Kenta Ueyama}

\address{
Department of Mathematics,
Faculty of Education,
Hirosaki University,
1 Bunkyocho, Hirosaki, Aomori 036-8560, JAPAN}
\email{k-ueyama@hirosaki-u.ac.jp} 



\thanks {
The first author was supported by JSPS Grant-in-Aid for Scientific Research (C) 25400037.
The second author was supported by JSPS Grant-in-Aid for Young Scientists (B) 15K17503.}

\maketitle

\begin{abstract}
Tilting objects play a key role in the study of triangulated categories.
A famous result due to Iyama and Takahashi asserts that the stable categories of graded maximal Cohen-Macaulay modules over quotient singularities have tilting objects.
This paper proves a noncommutative generalization of Iyama and Takahashi's theorem using noncommutative algebraic geometry.
Namely, if $S$ is a noetherian AS-regular Koszul algebra and $G$ is a finite 
group acting on $S$ such that $S^G$ is a ``Gorenstein isolated singularity", 
then the stable category $\uCM^{\ZZ}(S^G)$ of graded maximal Cohen-Macaulay modules has a tilting object.  
In particular, the category $\uCM^{\ZZ}(S^G)$ is triangle equivalent to the derived category of a finite dimensional algebra.
\end{abstract}


\section{Introduction} 

Triangulated categories are increasingly important in many areas of mathematics
including algebraic geometry and representation theory.
There are two major classes of triangulated categories,
namely, derived categories of abelian categories and stable categories of Frobenius categories.
For example, derived categories of coherent sheaves have been extensively studied in algebraic geometry,
and stable module categories of selfinjective algebras have been extensively studied in representation theory of finite dimensional algebras.

In the study of triangulated categories, tilting objects play a key role.
They often enable us to realize abstract triangulated categories as concrete derived categories of modules over algebras.
One of the remarkable results on the existence of tilting objects has been obtained by Iyama and Takahashi.

\begin{theorem} \textnormal{\cite[Theorem 2.7, Corollary 2.10]{IT}} \label{thm.IT}
Let $S=k[x_1, \dots, x_d]$ be a polynomial algebra over an algebraically closed field $k$ of characteristic $0$
such that $\deg x_i =1$ and $d\geq 2$.
Let $G$ be a finite subgroup of $\SL(d,k)$ acting linearly on $S$, and $S^G$ the fixed subalgebra of $S$.
Assume that $S^G$ is an isolated singularity.
If we define the graded $S^G$-module
$$ T := \bigoplus_{i=1}^{d} f_*\Omega_{S}^{i}k(i)$$
where $f:S^G\to S$ is the inclusion, then the stable category $\uCM^{\ZZ}(S^G)$ of graded maximal Cohen-Macaulay modules has a tilting object
$[T]_{\rm CM}$, where $[T]_{\rm CM}$ is the maximal direct summand of $T$ which is a graded maximal Cohen-Macaulay module. 
As a consequence, there exists a finite dimensional algebra $\Gamma$ of finite global dimension such that
$$\uCM^{\ZZ}(S^G) \cong \cD^{b}(\mod \Gamma).$$
\end{theorem}

The stable categories of graded maximal Cohen-Macaulay modules are crucial objects studied in representation theory of algebras (see \cite{IT}, \cite{AIR} etc.)
and also attract attention from the viewpoint of Kontsevich's homological mirror symmetry conjecture (see \cite{KST1}, \cite{KST2} etc.). 
The aim of the present paper is to generalize Theorem \ref{thm.IT} to the noncommutative case using noncommutative algebraic geometry.


For the rest, we basically follow the terminologies and notations in \cite[Section 1.1]{MU} (see subsection 1.1 below).
Throughout this paper, $k$ denotes an algebraically closed field.
Let $A$ be a graded algebra over $k$, and $G$ a finite subgroup of $\GrAut A$ such that $\fchar k$ does not divide $|G|$.
In this case, $kG$ is a semisimple algebra.
Two idempotent elements
\[ e:= \frac{1}{|G|}\sum_{g \in G} g, \quad  \textrm{and} \quad e' := 1-e  \]
of $kG$ 
play crucial roles in this paper.  Since $kG\subset A*G$, we often view $e, e'$ as idempotent elements of $A*G$.
Moreover, since $e(A*G)e \cong A^G$ as graded algebras, we usually identify $e(A*G)e$ with $A^G$.


In \cite{JZ}, J\o rgensen and Zhang introduced the notion of homological
determinant for a graded algebra automorphism $g$ of a noetherian AS-regular algebra $S$ over $k$, 
and proved that if $G$ is a finite subgroup of the homological special linear group
$$\HSL(S)= \{ g \in \GrAut S \mid \textrm{the homological determinant of $g$ is 1}\}$$
on $S$, then $S^G$ is a noncommutative Gorenstein algebra (called an AS-Gorenstein algebra).

In \cite{U}, the second author introduced a notion of graded isolated singularity for
noncommutative graded algebras, which agrees with the usual notion of isolated singularity if the algebra is commutative and generated in degree 1,
and found some nice properties of such algebras.
For a noetherian AS-regular algebra $S$ over $k$ and a finite subgroup $G\leq \GrAut S$,
it was shown in \cite{MU} that
the condition that $S*G/(e)$ is finite dimensional over $k$ is closely related to the noncommutative graded isolated singularity property of $S^G$.
More precisely, for $G \leq \HSL(S)$, it was proved that 
$S*G/(e)$ is finite dimensional over $k$ if and only if $S^G$ is a noncommutative graded isolated singularity and $S*G \cong \uEnd_{S^G}(S)$
(see \cite[Theorem 3.10]{MU}).


If $S=k[x_1, \dots, x_d]$ is the polynomial algebra and $G$ is a finite group acting on $S$, then $\Spec S^G\cong \AA^d/G$ is a quotient singularity.  Further, if $\deg x_i=1$ for all $i$, then $\tails S\cong \coh \PP^{d-1}$ the category of coherent sheaves on $\PP^{d-1}$, and it is well-known that $\cO_{\PP^{d-1}}, \cO_{\PP^{d-1}}(1), \dots, \cO_{\PP^{d-1}}(d-1)$ is a full strong exceptional sequence for $\cD^b(\coh \PP^{d-1})$ so that $\bigoplus _{i=0}^{d-1}\cO_{\PP^{d-1}}(i)$ is a tilting object for $\cD^b(\coh \PP^{d-1})$.  
Suppose that a finite group $G$ acts on a noetherian AS-regular algebra $S$ over $k$ of Gorenstein parameter $\ell$.  Since $S$ is commutative if and only if $S$ is a polynomial algebra, the fixed subalgebra $S^G$ is regarded as a noncommutative quotient singularity and the noncommutative projective scheme $X:=\Proj _{nc}S$ associated to $S$ is regarded as a quantum projective space.  
The inclusion map $f:S^G\to S$ induces a functor $f_*:\tails S\to \tails S^G$.  If $G$ is non-trivial, then $f_*\cO_X, f_*\cO_X(1), \dots , f_*\cO_X(\ell-1)$ is no longer an exceptional sequence for $\cD^b(\tails S^G)$, however, the following result shows that $\bigoplus _{i=0}^{\ell-1}f_*\cO_X(i)$ is a tilting object for $\cD^b(\tails S^G)$ if $S^G$ is an ``isolated singularity".      

\begin{theorem} \textnormal {\cite[Theorem 3.14]{MU}}
Let $S$ be a noetherian AS-regular algebra over $k$ of dimension $d \geq  2$
and of Gorenstein parameter $\ell$, and $G$ a finite subgroup of $\GrAut S$ such that $\fchar k$ does not divide $|G|$.
If 
$S*G/(e)$ is finite dimensional over $k$, then
$$\bigoplus^{\ell-1}_{i=0} f_*\cO_X(i)$$
is a tilting object in $\cD^b(\tails S^G)$ where $X:=\Proj _{nc}S$ and $f:S^G\to S$ is the inclusion.
\end{theorem}

There exists another full strong exceptional sequence $\Omega ^{d-1}_{\PP^{d-1}}(d-1), \dots , \Omega^1_{\PP^{d-1}}(1), \Omega^0_{\PP^{d-1}}=\cO_{\PP^{d-1}}$ for $\cD^b(\coh \PP^{d-1})$ so that $\bigoplus _{i=0}^{d-1}\Omega ^i_{\PP^{d-1}}(i)$ is a tilting object.   In the setting of the above theorem, if $G$ is non-trivial, then $f_*\Omega ^{d-1}_{X}(d-1), \dots, f_*\Omega^1_{X}(1), f_*\Omega^0_{X}$ is no longer an exceptional sequence for $\cD^b(\tails S^G)$, however, we will show in this paper that $\bigoplus _{i=0}^{d-1}f_*\Omega _X^i(i)$ is a tilting object for $\cD^b(\tails S^G)$ if $S$ is Koszul.   

\begin{theorem} {\rm (Theorem \ref{thm.ksg})} \label{Ithm1}
Let $S$ be a noetherian AS-regular Koszul algebra of dimension $d$ over $k$,
and $G$ a finite subgroup of $\GrAut S$ such that $\fchar k$ does not divide $|G|$.
If 
$S*G/(e)$ is finite dimensional over $k$, then
$$\bigoplus^{d-1}_{i=0} f_*\Omega _X^i(i)$$
is a tilting object in $\cD^b(\tails S^G)$ where $X:=\Proj _{nc}S$ and $f:S^G\to S$ is the inclusion.
As a consequence, there exists a finite dimensional algebra $\Lambda$ of finite global dimension such that
$$\cD^b(\tails S^G) \cong \cD^{b}(\mod \Lambda).$$
\end{theorem}

In the setting of the above theorem, the skew group algebra $S*G$ is an AS-regular Koszul algebra of dimension $d$ over $kG$ (Proposition \ref{prop.49} and Lemma \ref{lem5}),
and the Koszul dual algebra $(S*G)^!$ is a Frobenius Koszul algebra of Gorenstein parameter $-d$ (Proposition \ref{prop.n11}).  Thus one can define the stable category $ \ugrmod (S*G)^!$.
Then we have an equivalence $\overline{K}:\ugrmod (S*G)^! \to \cD^b(\tails S*G)$ called the BGG correspondence (Proposition \ref{prop.MS}).
If $S*G/(e)$ is finite dimensional over $k$, then we also have the equivalence $(-)e: \cD^b(\tails S*G) \to \cD^b(\tails S^G)$ (Proposition \ref{prop.amp}).
The key point of our proof is to show that under the equivalences
$$
\xymatrix{
\ugrmod (S*G)^! \ar[r]_{\overline{K}}^{\sim} &\cD^b(\tails S*G) \ar[r]^{\sim}_{(-)e} & \cD^b(\tails S^G),
}$$
the tilting object in $\ugrmod (S*G)^!$ which was obtained by Yamaura \cite{Y} corresponds to
the object $\bigoplus^{d-1}_{i=0} f_*\Omega _X^i(i)$
in $\cD^b(\tails S^G)$ 
(Lemma \ref{lem.Y}, Corollary \ref{cor.tsg} and  Theorem \ref{thm.ksg}).

We define a graded $kG$-$S*G$ bimodule $U$ by
$$ U := \bigoplus_{i=1}^{d} \Omega_{S*G}^{i}kG(i).$$Using Theorem \ref{Ithm1}, we will show the existence of a tilting object of the stable category $\uCM^{\ZZ}(S^G)$ if $S^G$ is a ``Gorenstein isolated singularity".
The main result of this paper is as follows.

\begin{theorem} {\rm (Theorem \ref{thm.CMtilt}, Theorem \ref{thm.main})} \label{Ithm2}
Let $S$ be a noetherian AS-regular Koszul algebra of dimension $d\geq 2$ over $k$,
and $G$ a finite subgroup of $\HSL(S)$ such that $\fchar k$ does not divide $|G|$.
If 
$S*G/(e)$ is finite dimensional over $k$, then
$$e'Ue $$
is a tilting object in $\uCM^{\ZZ}(S^G)$.
As a consequence, there exists a finite dimensional algebra $\Gamma=e'\Lambda e'$ of finite global dimension such that
$$\uCM^{\ZZ}(S^G) \cong \cD^{b}(\mod \Gamma).$$
\end{theorem}

If $S$ is a commutative AS-regular Koszul algebra of dimension $d$ over $k$, then $S=k[x_1, \dots, x_d]$ with $\deg x_i =1$ for all $i$.
In this case, 
\begin{itemize}
\item $\HSL(S)$ coincides with $\SL(d, k)$.
\item 
$S*G/(e)$ is finite dimensional over $k$ if and only if $S^G$ is a (graded) isolated singularity (see \cite[Corollary 3.11]{MU}).
\item $e'Ue = [T]_{\rm CM}$ (see \cite[Proof of Theorem 2.9]{IT}).
\end{itemize}
Thus it follows that our result is a generalization of Theorem \ref{thm.IT}.
However, our proof is different from the original one given in \cite{IT}.
Thanks to Theorem \ref{Ithm1}, we can give a more conceptual proof in terms of triangulated categories.  
The following is a flow diagram of objects which are essential for this paper.

\begin{align*}
\xymatrix@R=8pt@C=8pt{ 
\cD^b(\grmod S*G) &&&& \cD^b(\grmod S^G) &&&& 
\\
*++[F-:<3pt>]{U := \bigoplus_{i=1}^{d}\Omega^{i}_{S*G}kG(i) } 
&&&& Ue \cong {\bigoplus_{i=1}^{d}f_*\Omega^{i}_{S}k(i)} \ar@{}[r]**+++\frm<3pt>{-} |(0.7){\oplus\!\!>} &{e'Ue} \\
\\
\\
\\
\cD^b(\tails S*G)  &&&& \cD^b(\tails S^G) &&&& \cD_{\rm Sg}^{\rm gr}(S^G)\\
*++[F-:<3pt>]{ \pi U=\bigoplus _{i=0}^{d-1}\Omega _Y^i(i+1) } &&&& \pi Ue \cong \bigoplus^{d-1}_{i=0} f_*\Omega _X^i(i+1)\;\; \ar@{}[r]**+++\frm<3pt>{-} |(0.7){\oplus\!\!>} &{\pi e'Ue}
&&& *++[F-:<3pt>]{ \upsilon e'Ue }\\ 
\\
\\
\ugrmod {(S*G)^!} &&&& \cD^b(\tails S) &&&& \uCM^{\ZZ} {S^G}\\
\bigoplus_{i=0}^{d-1} \frac{(S*G)^!(i)}{(S*G)^!(i)_{\geq 1}} \ar@{}[d]**++\frm<3pt>{-} |{\rotatebox{90}{$:$}} &&&& \bigoplus _{i=0}^{d-1}\Omega _X^i(i+1)\ar@{}[d]**++\frm<3pt>{-} |{\rotatebox{90}{$:$}} &&&&
e'Ue  \ar@{}[d]**++\frm<3pt>{-} |{\rotatebox{90}{$:$}} \\
\textrm{Yamaura tilting object} &&&& \textrm{a tilting object} &&&& \textrm{our tilting object}\\ 
\ar@/^10truemm/@<11truemm> "1,1";"6,1"^(0.65){\pi}
\ar@/^10truemm/@<10truemm>@{<-} "6,1";"10,1"|(0.72){\txt{\tiny BGG correspondence 
\\ \scriptsize \; $\overline{K}$}}^(0.53){\rotatebox{-90}{$\sim$}}
\ar@/^10truemm/@<7truemm>@{<-} "6,5";"10,5"^(0.7){f_*}
\ar"1,1";"1,5"_{(-)e}
\ar"6,1";"6,5"_{(-)e}^{\sim}
\ar@/^11truemm/@<5truemm> "1,5";"6,5"^(0.65){\pi}
\ar@/^15truemm/"1,5";"6,9"^(0.6){\txt{\tiny Verdier localization \\ \scriptsize $\upsilon$ } }
\ar@{<-}@/^8truemm/@<4truemm> "6,9";"10,9"|(0.72){\txt{\tiny Buchweitz\\ \tiny equivalence}}^(0.53){\rotatebox{-90}{$\sim$}}
\ar@{_{(}->}"6,9";"6,5"_(0.45){\txt{\tiny Orlov \\ \tiny embedding }} ^(0.45){\Phi}
\ar@{|.>}@<-6truemm>@/_8truemm/ "2,1";"7,1"
\ar@{<.|}@<-6truemm>@/_8truemm/ "7,1";"11,1"
\ar@{<.|}@<-10truemm>@/_4truemm/ "7,5";"11,5"
\ar@{|.>}"2,1";"2,5"
\ar@{|.>}"7,1";"7,5"
\ar@{|.>}@<-11truemm>@/_4truemm/ "2,5";"7,5"
\ar@{|.>}@<1truemm>@/_4truemm/ "2,6";"7,6"
\ar@{|.>}@<0truemm>@/_4truemm/ "2,6";"7,9"
\ar@{|.>}"7,9";"7,6"
\ar@{<.|}@/_12truemm/ "7,9";"11,9"
}
\end{align*}
where $X :=\Proj _{nc}S$, $Y :=\Proj _{nc}S*G$, and $M \oplus\!\!\!\!> N$ means that $N$ is a direct summand of $M$.

In the last section of this paper, we will give an illustrative example.

\subsection{Terminologies and Notations}
We first introduce some basic terminologies and notations used in this paper.
Let $k$ be an algebraically closed field.
Unless otherwise stated, a graded algebra means an $\NN$-graded algebra $A =\bigoplus_{i\in\NN} A_i$ over $k$.
The group of graded $k$-algebra automorphisms of $A$ is denoted by
$\GrAut A$. 
We denote by $\GrMod A$ the category
of graded right $A$-modules, and by $\grmod A$ the full subcategory consisting of
finitely presented modules.  Note that if $A$ is graded right coherent, then $\grmod A$ is an abelian category.   
Morphisms in $\GrMod A$ are right $A$-module homomorphisms of degree zero.
Graded left $A$-modules are identified with graded $A^o$-modules where $A^o$ is the opposite graded algebra of $A$.  For
$M \in \GrMod A$ and $n \in \ZZ$, we define $M_{\geq n}$ = $\bigoplus_{i\geq n} M_i \in \GrMod A$, and $M(n) \in
\GrMod A$ by $M(n) = M$ as an ungraded right $A$-module with the new grading
$M(n)_i = M_{n+i}$.
The rule $M \mapsto  M(n)$ is a
$k$-linear autoequivalence for $\GrMod A$ and $\grmod A$, 
called the shift functor.
For $M,N \in \GrMod A$, we write the vector space
$\Ext^i_A(M,N) := \Ext^i_{\GrMod A}(M,N)$
and the graded vector space
$$\uExt^i_A(M,N) :=\bigoplus_{n\in\ZZ}\Ext^i_A(M,N(n)).$$
We denote by $D$ the $k$-vector space dual.  For a graded right (resp. left) $A$-module $M$,
we also denote by $DM:=\uHom _k(M, k)$ the graded $k$-vector space
dual of $M$ by abuse of notation. Note that $DM$ has a graded
left (resp. right) $A$-module structure.
We say that $M \in \GrMod A$ is torsion if, for any $m \in M$, there exists $n \in \NN$
such that $mA_{\geq n} = 0$. We denote by $\Tors A$ the full subcategory of $\GrMod A$
consisting of torsion modules.
We write $\Tails A$ for the quotient category $\GrMod A / \Tors A$.
The quotient functor is denoted by $\pi :\GrMod A \to \Tails A$. 
The objects in $\Tails A$ will be denoted by script letters, like $\cM = \pi M$.  
If $A$ is graded right coherent, 
then 
we define $\tails A := \grmod A/\tors A$ where $\tors A = \Tors A \cap \grmod A$ is the full subcategory of $\grmod A$ consisting of finite dimensional modules.  We call $\tails A$ the noncommutative projective scheme associated to $A$ since $\tails A\cong \coh (\Proj A)$ the category of coherent sheaves on $\Proj A$ if $A$ is commutative and generated in degree 1. 
The shift functor on $\GrMod A$ induces an autoequivalence $(n) :\cM \mapsto \cM(n)$ for
$\Tails A$ and $\tails A$, also call the shift functor.
For $\cM,\cN \in \ \Tails A$, we write the vector space
$\Ext^i_{\cA}(\cM,\cN) := \Ext^i_{\Tails A}(\cM,\cN)$
and the graded vector space
$$\uExt^i_{\cA}(\cM,\cN) :=\bigoplus_{n\in\ZZ}\Ext^i_{\cA}(\cM,\cN(n)).$$
We remark that, in many papers (eg. \cite{IT}), $\uHom$ and $\uEnd$ denote $\Hom$ and $\End$ in the stable categories, however, in this paper, $\uHom$ and $\uEnd$ always mean graded $\Hom$ and graded $\End$ as defined above, following the tradition of noncommutative algebraic geometry starting from \cite{AZ}.

\section{Algebras to Study} 

In this section, we will give definitions and basic properties of algebras to study in this paper.  Some of the results in this section and the next section are essentially not new but rather slight generalizations of results in \cite {Ma2} and \cite {MS} etc.  We will give proofs for some of such results to make this paper as self-contained as possible for the reader who is not an expert.  

\subsection{Koszul Algebras} 

\begin{definition} \label{def.Kos} 
Let $A$ be a graded algebra.  A linear resolution of $M\in \GrMod A$ is a graded projective resolution of $M$
$$\cdots \to P^2\to P^1\to P^0\to M\to 0$$
where each $P^i$ is a 
graded projective right $A$-module generated in degree $i$.  
 
A graded algebra $A$ is called Koszul if 
\begin{enumerate}
\item{} $A$ is locally finite, that is, $\dim _kA_i<\infty$ for all $i\in \NN$,
\item{} $A_0$ is a semisimple algebra,  and 
\item{} $A_0:=A/A_{\geq 1}\in \GrMod A$ has a linear resolution.
\end{enumerate} 
\end{definition} 

In many literature, the definition of a Koszul algebra $A$ requires for $A_0$ to be a finite direct product of $k$.  On the other hand, if $A$ is a Koszul algebra defined as above, then $A_0$ is a finite dimensional semisimple algebra over an algebraically closed field $k$, so $A_0$ is isomorphic to a finite direct product of full matrix algebras over $k$.
Thus a Koszul algebra in this paper is more general than that in many literature, but more restrictive than that in \cite {BGS}.
In particular, $A_0$ is a symmetric algebra, that is, $D(A_0)\cong A_0$ as $A_0$-$A_0$ bimodules,
so we have the following result for our Koszul algebras.   

\begin{lemma} \label{lem.DH}
Let $A$ be a Koszul algebra.  
\begin{enumerate}
\item{} For every finite dimensional $A_0$-$A_0$ bimodule $M$, 
$\Hom_{A_0}(M, A_0)\cong \Hom_{A_0^o}(M, A_0)\cong D(M)$
as $A_0$-$A_0$ bimodules.
\item{} For every locally finite $A$-$A_0$ bimodule $M$, 
$\uHom_{A_0}(M, A_0)\cong D(M)$ as graded $A_0$-$A$ bimodules .
\item{} For every locally finite $A_0$-$A$ bimodule $M$, 
$\uHom_{A_0^o}(M, A_0)\cong D(M)$ as graded $A$-$A_0$ bimodules .
\end{enumerate}
\end{lemma} 

Due to the above lemma, we identify the right dual and the left dual, which were carefully distinguished in \cite {BGS} (see \cite [Section 2.7]{BGS}). Another property of our Koszul algebra is as follows.  A graded algebra $A$ is called right (resp. left) finite if
$A_i$ are finitely generated as right (resp. left) $A_0$-modules for all $i\in \NN$.
It is easy to see that a graded algebra $A$ is locally finite if and only if $A$ is right (or left) finite and $\dim _kA_0<\infty$, so our Koszul algebra is both left and right finite.  Our Koszul algebra enjoys the following usual properties of a Koszul algebra.  

\begin{lemma} \textnormal {\cite [Proposition 2.2.1]{BGS}} \label{lem.kd2} 
Let $A$ be a graded algebra.  Then $A$ is Koszul if and only if $A^o$ is Koszul.
\end{lemma} 


\begin{definition} For a graded algebra $A$, we define the graded algebra 
$$A^!:=E_A(A_0):=\bigoplus _{i\in \NN}\uExt_A^i(A_0, A_0),$$
and the functor $E_A:\GrMod A\to \GrMod (A^!)^o$ by 
$$E_A(M):=\bigoplus _{i\in \NN}\uExt^i_A(M, A_0).$$
\end{definition} 


\begin{lemma} \textnormal {\cite [Proposition 2.9.1, Theorem 2.10.2]{BGS}}
If $A$ is a Koszul algebra, then 
\begin{enumerate}
\item{} $A^!$ is also a Koszul algebra, called the Koszul dual of $A$, and 
\item{} $(A^!)^!\cong A$ as graded algebras.
\end{enumerate} 
\end{lemma} 

If $A$ is a Koszul algebra, then it is a quadratic algebra by \cite [Corollary 2.3.3]{BGS}, so we may write $A=T_{A_0}(A_1)/(R)$ for some $R\subset A_1\otimes_{A_0}A_1$.  

\begin{lemma} \textnormal{\cite [Theorem 2.6.1]{BGS}} \label{lem.linbgs}
If $A=T_{A_0}(A_1)/(R)$ is a Koszul algebra where $R\subset A_1\otimes _{A_0}A_1$, 
then the linear resolution of $A_0$ over $A$ is given by $(K^i_i\otimes _{A_0}A, d^i)$ where 
$$K^i_i:=\bigcap _{s+t+2=i}A_1^{\otimes _{A_0}s}\otimes _{A_0}R\otimes _{A_0}A_1^{\otimes _{A_0}t}\subset A_1^{\otimes _{A_0}i}$$
and $d^i$ is the restriction of the map 
$$A_1^{\otimes _{A_0}i}\otimes _{A_0}A\to A_1^{\otimes _{A_0}i-1}\otimes _{A_0}A; \;  v_1\otimes \cdots \otimes v_{i-1}\otimes v_i\otimes a\mapsto v_1\otimes \cdots \otimes v_{i-1}\otimes v_ia.$$
\end{lemma} 

The full subcategory of $\GrMod A$ consisting of modules having linear resolutions $P^{\bullet}$ such that each $P^i$ is a 
finitely generated 
graded projective right $A$-module (generated in degree $i$) is denoted by $\lin A$. 

\begin{lemma} \label{lem.lin} 
If $A$ is a Koszul algebra, then $A_0\in \lin A$ and $A_0\in \lin A^o$. 
\end{lemma}

\begin{proof} 
By Lemma \ref{lem.linbgs}, the linear resolution $K^{\bullet}$ of $A_0\in \GrMod A$ is of the form $K^i=K^i_i\otimes _{A_0}A$ where each $K^i_i$ is an $A_0$ subbimodule of $A_1^{\otimes _{A_0}i}$.  Since $\dim _kA_1<\infty$, $\dim _kK^i_i<\infty$. It follows that $K^i_i$ is finitely generated as a right $A_0$-module, so $K^i$ is finitely generated as a graded right $A$-module, hence $A_0\in \lin A$.  By Lemma \ref{lem.kd2}, $A_0\in \lin A^o$. 
\end{proof} 

\begin{definition} \label{def.mr} 
Let $A$ be a graded algebra and $M\in \GrMod A$.  A graded projective resolution $(P^{\bullet}, d)$ is called minimal if $d^i\otimes _AA_0=0$ for every $i\in \NN^+$.
\end{definition} 

If $A$ is a graded algebra such that $A_0$ is semisimple, then it is known that a minimal projective resolution of $M\in \GrMod A$ is unique up to isomorphism. 
Clearly, every linear resolution is a minimal resolution, so we have the following result. 

\begin{lemma}
Let $A$ be a Koszul algebra.  A linear resolution of $M\in \GrMod A$ is unique up to isomorphism if it exists. 
\end{lemma}

\begin{definition}
Let $A$ be a Koszul algebra and $(P^{\bullet}, d)$ the linear resolution of $M\in \lin A$.  The $i$-th syzygy of $M$ is defined by $\Omega ^iM:=\Im d^i$ for $i\in \NN^+$.  
\end{definition}

\begin{remark} \label{rem.left0}
Let $A$ be a Koszul algebra.  Since the linear resolution $(K^{\bullet}, d)$ of $A_0$ is a complex of graded $A_0$-$A$ bimodules
(see \cite[Remark (1) to Definition 1.2.1]{BGS}), $\Omega ^i A_0$ has a graded $A_0$-$A$ bimodule structure.
\end{remark}

\subsection{Skew Group Algebras} 

\begin{definition} 
Let $R$ be an algebra and $G\leq \Aut R$ a subgroup.  We define the skew group algebra of $R$ by $G$  
by $R*G=R\otimes_kkG$ as a $k$-vector space with the multiplication 
$$(a\otimes g)\cdot (b\otimes h)=ag(b)\otimes gh.$$ 
An element $a\otimes g\in R*G$ is usually denoted by $a*g$. 

Similarly, we define $G*R=kG\otimes_kR$ as a $k$-vector space with the multiplication 
$$(g\otimes a)\cdot (h\otimes b)=gh\otimes h^{-1}(a)b.$$ 
\end{definition} 

The proof of the following lemma is left to the reader.

\begin{lemma} \label{lem.sga} \label{lem.ags} 
Let $R$ be an algebra and $G\leq \Aut R$ a subgroup. Then the following hold:
\begin{enumerate}
\item $(R*G)^o\cong R^o*G$ as algebras.
\item $G*R\cong R*G^o$ as algebras. 
\end{enumerate}  
\end{lemma} 

Let $A$ be a graded algebra, and $G\leq \GrAut A$ a finite subgroup.  If $\operatorname {char} k$ does not divide $|G|$, then we define $e:=\frac{1}{|G|}\sum_{g\in G}g\in kG\subset A*G$.  By \cite {MU}, $A$ has a graded right $A*G$-module structure by $c\cdot (a*g):=g^{-1}(ca)$, and $A\cong e(A*G)$ as graded right $A*G$-modules.  Since $A_{\geq i}$ is a submodule of $A$ as a graded right $A*G$-module, $A_0:=A/A_{\geq 1}\cong e(A*G)_0\cong e(A_0*G)$ has a graded right $A*G$-module structure. 
We also define a graded left $A*G$-module structure on $A$ by $(a*g)\cdot c:=ag(c)$.
Since $A_{\geq i}$ is a submodule of $A$ as a graded left $A*G$-module, $A_0:=A/A_{\geq 1}\cong (A*G)_0e\cong (A_0*G)e$ has a graded left $A*G$-module structure as well.  
 
If $M$ is a graded left $A*G$-module and $V$ is a left $kG$-module, then we define the graded left $A*G$-module structure on $M\otimes _kV$ by $(a*g)\cdot (m\otimes v):=(a*g)m\otimes gv$ (cf. \cite {Ma2}).   Since $A$ has a graded left $A*G$-module structure, $A\otimes _kkG$ has a graded left $A*G$-module structure by 
$(a*g)(b\otimes h)=(a*g)b\otimes gh=ag(b)\otimes gh$.
It follows that $A\otimes _kkG\cong A*G$ and 
$A_0\otimes _kkG\cong (A*G)_0$ as graded left $A*G$-modules.
We can define the action of $G^o$ on $A^!$ as explained in \cite{RRZ2}.
We remark that this action is the ``inverse'' of the action in \cite {Ma2}.

The next proposition was essentially proved in \cite {Ma2}, however, the definition of a Koszul algebra given in \cite {Ma2} is slightly different from ours, so we will provide our own proof.  

\begin{proposition} \label{prop.49} \textnormal{(cf. \cite [Theorem 10, Theorem 14]{Ma2})} 
Let $A$ be a graded
algebra 
and $G\leq \GrAut A$ a finite subgroup such that $\operatorname {char} k$ does not divide $|G|$. 
If $A_0$ has a graded projective resolution consisting of finitely generated graded projective right $A*G$-modules, then the following hold: 
\begin{enumerate} 
\item{} $(A*G)^!\cong G*A^!$ as graded algebras. 
\item{} $A$ is Koszul if and only if $A*G$ is Koszul. 
\end{enumerate}
\end{proposition} 

\begin{proof} (1) If $A_0$ has a graded projective resolution consisting of finitely generated graded projective left $A*G$-modules,       
then we have an isomorphism 
\begin{align*}
\uExt ^i_{A^o}(A_0, A_0)\otimes _kkG & 
\cong \uExt^i_{(A*G)^o}(A_0\otimes _kkG, A_0\otimes _kkG) \\
& 
\cong \uExt^i_{(A*G)^o}((A*G)_0, (A*G)_0)
\end{align*}
of graded vector spaces for each $i\in \NN$ by \cite[Lemma 8, Proposition 9]{Ma2} 
(one can check that the isomorphism $\theta$ in \cite [Lemma 8]{Ma2} is given by
\[ \theta(f \otimes w)(p \otimes g) = (g^{-1} \star f)(p) \otimes gw^{-1}  \]
in our setting where $\star$ means the group action, and can also check that it preserves the grading).
Furthermore we have an isomorphism 
$$E_{A^o}(A_0)^o* G^{o}\cong E_{(A*G)^o}((A*G)_0)^o$$
as graded algebra proved by the same arguments as in \cite [Lemma 10]{Ma2}.   Now we replace $A$ by $A^o$.  Since $(A^o*G)^o\cong A*G$ by Lemma \ref{lem.sga}, 
if $A_0$ has a graded projective resolution consisting of finitely generated graded projective right $A*G$-modules,       
then we have an isomorphism 
$$\uExt ^i_{A}(A_0, A_0)\otimes _kkG \cong \uExt^i_{(A^o*G)^o}((A*G)_0, (A*G)_0)\cong \uExt_{A*G}^i((A*G)_0, (A*G)_0)$$
of graded vector spaces for each $i\in \NN$, and an isomorphism 
$$(A*G)^!=E_{A*G}((A*G)_0)\cong E_{(A^o*G)^o}((A*G)_0)\cong (E_A(A_0)^o*G^o)^o\cong E_A(A_0)*G^o=G*A^!$$ 
as graded algebras by Lemma \ref{lem.sga}. 

(2) Since $(A*G)_i\cong A_i\otimes _kkG$ as vector spaces, $A$ is locally finite if and only if $A*G$ is locally finite.  
Since 
$\operatorname {char} k$ does not divide $|G|$, $\gldim (A*G)_0=\gldim A_0*G=\gldim A_0$, so $A_0$ is semisimple if and only if $(A*G)_0$ is semisimple. 

Since 
\begin{align*}
\Ext^i_{A*G}((A*G)_0, (A*G)_0(j)) & \cong \uExt^i_{A*G}((A*G)_0, (A*G)_0)_j \\
& \cong (\uExt ^i_{A}(A_0, A_0)\otimes _kkG)_j \\
& \cong \uExt ^i_{A}(A_0, A_0)_j\otimes _kkG \\
& \cong \Ext^i_{A}(A_0, A_0(j))\otimes _kkG,
\end{align*}
$(A*G)_0\in \lin A*G$ if and only if $\Ext^i_{A*G}((A*G)_0, (A*G)_0(j))=0$ unless $i+j=0$ if and only if $\Ext^i_{A}(A_0, A_0(j))=0$ unless $i+j=0$ if and only if $A_0\in \lin A$ 
by \cite[Proposition 2.14.2]{BGS}.  
\end{proof} 



Let $A$ be a graded algebra, $G\leq \GrAut A$ a finite subgroup such that $\operatorname {char} k$ does not divide $|G|$, and $e:=\frac{1}{|G|}\sum _{g\in G}g\in kG$ an idempotent.  We define a graded right $A$-module structure on $(A*G)e$ by $(a*g)e\cdot b=(ab*g)e$. Since 
$$((a*g)e\cdot b)\cdot c=(ab*g)e\cdot c=(abc*g)e=(a*g)e\cdot bc,$$
it is a well-defined graded right $A$-module structure.

\begin{lemma} \label{lem.phi} If $A$ is a graded algebra, $G\leq \GrAut A$ is a finite subgroup such that $\operatorname {char} k$ does not divide $|G|$, and $e:=\frac{1}{|G|}\sum _{g\in G}g\in kG$, then the map $\phi:(A*G)e\to A$ defined by $\phi((a*g)e)=a$ is an isomorphism in $\GrMod A$. 
\end{lemma} 

\begin{proof} Define a map $\psi:A\to (A*G)e$ by $\psi(a)=(a*\e)e$ where $\e\in G$ is the identity.  Since 
$$\psi(ab)=(ab*\e)e=(a*\e)e\cdot b=\psi (a)\cdot b,$$
$\psi$ is a graded right $A$-module homomorphism.  By \cite [Lemma 1.4]{MU}, $\psi$ is bijective, so $\psi:A\to (A*G)e$ is an isomorphism of graded right $A$-modules.  Since $\phi(\psi(a))=\phi ((a*\e)e)=a$ for every $a\in A$,  
$\phi=\psi^{-1}:(A*G)e\to A$ is an isomorphism of graded right $A$-modules. 
\end{proof}

\begin{lemma} \label{lem.ess} 
If $A$ is a graded algebra and $G\leq \GrAut A$, then the map
$\varphi :A_1^{\otimes _ki}\otimes _kA*G\to (A*G)_1^{\otimes _{kG}i}\otimes _{kG}A*G$ defined by $$\varphi(v_1\otimes\cdots \otimes v_i\otimes (a*g))=(v_1*\e)\otimes \cdots \otimes (v_i*\e)\otimes (a*g)$$ 
is an isomorphism in $\GrMod A*G$ where $\e\in G$ is the identity.
\end{lemma} 

\begin{proof} 
Define a map 
$\psi :(A*G)_1^{\otimes _{kG}i}\otimes _{kG}A*G\to A_1^{\otimes _ki}\otimes _kA*G$
by 
$$\psi ((v_1*g_1)\otimes (v_2*g_2)\otimes \cdots \otimes (v_i*g_i)\otimes (a*g))=v_1\otimes g_1(v_2)\otimes \cdots \otimes g_1\cdots g_{i-1}(v_i)\otimes (g_1\cdots g_i(a)*g_1\cdots g_ig).$$  For $(v_1*g_1)\otimes (v_2*g_2)\otimes (a*g)\in (A*G)_1\otimes _{kG}(A*G)_1\otimes _{kG}A*G$ and $h\in G$, \begin{align*}
\psi ((v_1*g_1)\cdot h\otimes (v_2*g_2)\otimes (a*g)) 
& =\psi ((v_1*g_1)(1*h)\otimes (v_2*g_2)\otimes (a*g)) \\
& =\psi ((v_1*g_1h)\otimes (v_2*g_2)\otimes (a*g)) \\
& =v_1\otimes g_1h(v_2)\otimes (g_1hg_2(a)*g_1hg_2g) \\
& =\psi ((v_1*g_1)\otimes (h(v_2)*hg_2)\otimes (a*g)) \\
& =\psi ((v_1*g_1)\otimes (1*h)(v_2*g_2)\otimes (a*g)) \\
& =\psi ((v_1*g_1)\otimes h\cdot (v_2*g_2)\otimes (a*g)), 
\end{align*}
so we can see that $\psi$ is well-defined.  Moreover, since 
\begin{align*}
(v_1*g_1)\otimes (v_2*g_2)\otimes (a*g)
& =(v_1*\e)(1*g_1)\otimes (v_2*g_2)\otimes (a*g) \\
& =(v_1*\e)\cdot g_1\otimes (v_2*g_2)\otimes (a*g) \\
& =(v_1*\e)\otimes g_1\cdot (v_2*g_2)\otimes (a*g) \\
& =(v_1*\e)\otimes (1*g_1)(v_2*g_2)\otimes (a*g) \\
& =(v_1*\e)\otimes (g_1(v_2)*g_1g_2)\otimes (a*g) \\
& =(v_1*\e)\otimes (g_1(v_2)*\e)(1*g_1g_2)\otimes (a*g) \\
& =(v_1*\e)\otimes (g_1(v_2)*\e)\cdot g_1g_2\otimes (a*g) \\
& =(v_1*\e)\otimes (g_1(v_2)*\e)\otimes g_1g_2\cdot (a*g) \\
& =(v_1*\e)\otimes (g_1(v_2)*\e)\otimes (1*g_1g_2)(a*g) \\
& =(v_1*\e)\otimes (g_1(v_2)*\e)\otimes (g_1g_2(a)*g_1g_2g),  
\end{align*}
we can show that 
\begin{align*}
& \varphi (\psi((v_1*g_1)\otimes (v_2*g_2)\otimes \cdots \otimes (v_i*g_i)\otimes (a*g))) \\
= & \varphi (v_1\otimes g_1(v_2)\otimes \cdots \otimes g_1\cdots g_{i-1}(v_i)\otimes (g_1\cdots g_i(a)*g_1\cdots g_ig)) \\
= & (v_1*\e)\otimes (g_1(v_2)*\e)\otimes \cdots \otimes (g_1\cdots g_{i-1}(v_i)*\e)\otimes (g_1\cdots g_i(a)*g_1\cdots g_ig) \\
= & (v_1*g_1)\otimes (v_2*g_2)\otimes \cdots \otimes (v_i*g_i)\otimes (a*g). 
\end{align*}
On the other hand, for $v_1\otimes v_2\otimes \cdots \otimes v_i\otimes (a*g)\in A_1^{\otimes _ki}\otimes _kA*G$, 
\begin{align*}
\psi (\varphi (v_1\otimes v_2\otimes \cdots \otimes v_i\otimes (a*g))) & = \psi ((v_1*\e)\otimes (v_2*\e)\otimes \cdots \otimes (v_i*\e)\otimes (a*g)) \\
& =v_1\otimes v_2\otimes \cdots \otimes v_i\otimes (a*g),
\end{align*}
so $\psi$ is the inverse of $\varphi$, hence $\varphi$ is a bijection.  Clearly, $\varphi$ is a graded right $A*G$-module homomorphism, so the result follows.   
\end{proof}

Let $A$ be a graded algebra, $G\leq \GrAut A$ a finite subgroup such that $\operatorname {char} k$ does not divide $|G|$, and $e:=\frac{1}{|G|}\sum _{g\in G}g\in kG$.  Since $(A*G)e$ is projective as a graded left $A*G$-module, and $e(A*G)e\cong A^G$ as graded algebras, we have an exact functor $(-)e:=-\otimes _{A*G}(A*G)e:\GrMod A*G\to \GrMod A^G$.  On the other hand, the inclusion map $f:A^G\to A$ induces an exact functor $f_*:\GrMod A\to \GrMod A^G$.  In fact, a complex $C^{\bullet}$ of graded right $A$-modules is exact if and only if the complex $f_*C^{\bullet}$ of graded right $A^G$-modules is exact.  

\begin{theorem} \label{thm.cor} 
Let $A$ be a connected graded Koszul algebra (that is, $A_0=k$), $G\leq \GrAut A$ a finite subgroup such that $\operatorname {char} k$ does not divide $|G|$, $e:=\frac{1}{|G|}\sum _{g\in G}g\in kG$, and $f:A^G\to A$ the inclusion map.  If $A*G$ is a Koszul algebra, then $(\Omega^i_{A*G}kG)e\cong f_*\Omega _A^ik$ in $\GrMod A^G$ for every $i\in \NN^+$.
\end{theorem} 

\begin{proof} 
Since $A*G$ is Koszul such that $(A*G)_0=kG$, 
we may write $A*G=T_{kG}((A*G)_1)/(\overline R)$ for some $\overline R\subset (A*G)_1\otimes _{kG}(A*G)_1$.  By Lemma \ref{lem.linbgs}, the linear resolution of $kG$ over $A*G$ is given by $(\overline K^i_i\otimes _{kG}A*G, \overline d^i)$ where 
$$\overline K^i_i:=\bigcap _{s+t+2=i}(A*G)_1^{\otimes _{kG}s}\otimes _{kG}\overline R\otimes _{kG}(A*G)_1^{\otimes _{kG}t}\subset (A*G)_1^{\otimes _{kG}i}$$
and $\overline d^i$ is the restriction of the map 
$$\begin{CD}
(v_1*\e)\otimes \cdots \otimes (v_{i-1}*\e)\otimes (v_i*\e)\otimes (a*g) & \mapsto & 
(v_1*\e)\otimes \cdots \otimes (v_{i-1}*\e)\otimes (v_i*\e)(a*g) \\
& & =(v_1*\e)\otimes \cdots \otimes (v_{i-1}*\e)\otimes (v_ia*g) \\
(A*G)_1^{\otimes _{kG}i}\otimes _{kG}A*G @>>> (A*G)_1^{\otimes _{kG}i-1}\otimes _{kG}A*G \\
@A\varphi A\cong A @A\varphi A\cong A \\
A_1^{\otimes _ki}\otimes _kA*G @>>> A_1^{\otimes _ki-1}\otimes _kA*G \\
v_1\otimes \cdots \otimes v_{i-1}\otimes v_i\otimes (a*g) & \mapsto 
& v_1\otimes \cdots \otimes v_{i-1}\otimes (v_ia*g)
\end{CD}$$
(using $\varphi$ in Lemma \ref{lem.ess}, the above is a commutative diagram in $\GrMod A*G$). 

Since $(A*G)e\cong A$ as graded $k$-$A$ bimodules by Lemma \ref{lem.phi}, 
$$(A_1^{\otimes _ki}\otimes _kA*G)e\cong A_1^{\otimes _ki}\otimes _k(A*G)e\cong A_1^{\otimes _ki}\otimes _kA$$ has a structure of a graded right $A$-module via $\phi$ in Lemma \ref{lem.phi}.  By the commutative digram 
$$\begin{CD}
v_1\otimes \cdots \otimes v_{i-1}\otimes v_i\otimes (a*g)e & \mapsto & v_1\otimes \cdots \otimes v_{i-1}\otimes (v_ia*g)e \\
A_1^{\otimes _ki}\otimes _k(A*G)e @>>> A_1^{\otimes _ki-1}\otimes _k(A*G)e \\
@V\id \otimes \phi V\cong V @V\id\otimes \phi V\cong V \\
A_1^{\otimes _ki}\otimes _kA & \to & A_1^{\otimes _ki-1}\otimes _kA \\ 
v_1\otimes \cdots \otimes v_{i-1}\otimes v_i\otimes a & \mapsto & v_1\otimes \cdots \otimes v_{i-1}\otimes v_ia,
\end{CD}$$
$\overline d^ie$ can be viewed as a graded right $A$-module homomorphism.

Since $(-)e:\GrMod A*G\to \GrMod A^G$ is an exact functor and 
$$(kG)e=(A*G)_0e=((A*G)e)_0\cong A_0=k$$
in $\GrMod A^G$,  
$$\begin{CD}
\cdots @>{\overline d^3e}>> (\overline K^2_2\otimes _{kG}A*G)e @>{\overline d^2e}>> (\overline K^1_1\otimes _{kG}A*G)e @>{\overline d^1e}>> (A*G)e\to (kG)e\cong k\to 0
\end{CD}$$
is an exact sequence in $\GrMod A^G$.  Since it has a structure of a complex of graded right $A$-modules as we have seen above and every differential is of degree 1, 
it can be viewed as the linear resolution of $k$ over $A$ (that is, the image of the linear resolution of $k$ over $A$ under the functor $f_*:\GrMod A\to \GrMod A^G$ coincides with the above exact sequence).   
Since $(-)e:\GrMod A*G\to \GrMod A^G$ and $f_*:\GrMod A\to \GrMod A^G$ are exact functors, $(\Omega^i_{A*G}kG)e\cong f_*\Omega _A^ik$ in $\GrMod A^G$ for every $i\in \NN^+$.
\end{proof} 

\begin{remark} Let $(\overline K^i:=\overline K^i_i\otimes _{kG}A*G, \overline d^i)$ be the linear resolution of $kG$ over $A*G$, and $(K^i:=K^i_i\otimes _kA, d^i)$ the linear resolution of $k$ over $A$.   
Although $(A*G)e$ has a graded left $A*G$-module structure and a graded right $A$-module structure, it is not a graded $A*G$-$A$ bimodule, so 
$$(-)e:=-\otimes _{A*G}(A*G)e:\GrMod A*G\to \GrMod A$$ 
is not a functor, 
hence it is not clear if $(\overline K^ie, \overline d^ie)$ has a structure of a complex of graded right $A$-modules.   
We thank the referee for pointing this out.   
The above proof claims the following:  Since$\overline K^i:=\overline K^i_i\otimes _{kG}A*G\cong K^i_i\otimes _kA*G$ as graded right $A*G$-modules, $\overline K^ie
\cong (K^i_i\otimes _kA*G)e\cong K^i_i\otimes _k(A*G)e$ as graded right $A^G$-modules.  
On the other hand, since $(A*G)e\cong A$ as graded $k$-$A$ bimodules by Lemma \ref{lem.phi}, $\overline K^ie\cong K^i_i\otimes _k(A*G)e\cong K^i_i\otimes _kA=:K^i$ has a structure of a graded right $A$-module, and we have shown above that the map $\overline d^ie$ coincides with the map $d^i$ under these identifications.  
\end{remark}

\subsection{AS-regular Algebras and Graded Frobenius Algebras} 

\begin{definition} \textnormal{\cite [Definition 3.1]{MM}} 
A locally finite graded algebra $A$ is called 
AS-regular over $A_0$ of dimension $d$ and of Gorenstein parameter $\ell$ if the following conditions are satisfied: 
\begin{enumerate}
\item{} $\gldim A_0<\infty$.
\item{} 
$\gldim A=d$. 
\item{} $A$ satisfies Gorenstein condition over $A_0$, that is, 
$$\uExt^i_A(A_0, A)\cong \begin{cases} D(A_0)(\ell) & \textnormal { if $i=d$, } \\
0 & \textnormal { if $i\neq d$ } \end{cases}$$
in $\GrMod A$ and in $\GrMod A^o$.  
\end{enumerate}
\end{definition}

\begin{remark} By \cite [Corollary 3.7]{MM}, a graded algebra $A$ is AS-regular over $A_0$ of dimension $d$ and of Gorenstein parameter $\ell$ if and only if $A^o$ is AS-regular over $A^o_0$ of dimension $d$ and of Gorenstein parameter $\ell$. 
\end{remark} 

\begin{lemma} \label{lem5} \textnormal{\cite [Corollary 3.6]{MU}} 
If $S$ is a noetherian AS-regular algebra over $k$ of dimension $d$ and of Gorenstein parameter $\ell$, and $G\leq \GrAut S$ is a finite subgroup such that $\operatorname {char} k$ does not divide $|G|$, then $S*G$ is a noetherian AS-regular algebra over $(S*G)_0=kG$ of dimension $d$ and of Gorenstein parameter $\ell$.  
\end{lemma}

\begin{definition} \textnormal{\cite [Definition 12]{Ma2}} 
A locally finite graded algebra $A$ is called generalized AS-regular of dimension $d$ if the following conditions are satisfied: 
\begin{enumerate}
\item{} $A_iA_j=A_{i+j}$ for all $i, j\in \NN$.  
\item{} $\gldim A=d$. 
\item{} Every simple graded right $A$ module has projective dimension $d$.
\item{} For every simple graded right $A$-module $S$, $\uExt^i_A(S, A)=0$ for $i\neq d$. 
\item{} The functors 
$$\uExt^d_A(-, A):\GrMod A \longleftrightarrow \GrMod A^o:\uExt^d_{A^o}(-, A)$$ 
induce a bijection between the set of all simple graded right $A$-modules and the set of all simple graded left $A$-modules. 
\end{enumerate}
\end{definition} 

\begin{remark} In \cite{Ma2}, Martinez-Villa called the algebra defined as above generalized Auslander regular.  
His definition requires $\operatorname {sgldim} A=d$ instead of $\gldim A=d$, but they are the same by \cite [Proposition 2.7]{MM}.     
\end{remark} 

\begin{lemma} 
\label{lem.gas} 
Every AS-regular algebra $A$ generated by $A_1$ over $A_0$ is generalized AS-regular.  
\end{lemma}

\begin{proof} Since $A$ is generated by $A_1$ over $A_0$, $A_iA_j=A_{i+j}$ for all $i, j\in \NN$.  By \cite [Proposition 3.6]{MM}, every simple graded right $A$ module has projective dimension $d$.  The rest of conditions follows from \cite [Theorem 3.17]{MM}.  
\end{proof} 

\begin{definition} A locally finite graded algebra $A$ is called graded Frobenius of Gorenstein parameter $-\ell$ if $D(A)\cong A(\ell)$ in $\GrMod A$ (or equivalently in $\GrMod A^o$). 
\end{definition} 

\begin{lemma} \label{lem.gF2}
If $A$ is a graded Frobenius algebra of Gorenstein parameter $-\ell$, and $G\leq \GrAut A$ is a finite subgroup,
then $A*G$ is a graded Frobenius algebra of Gorenstein parameter $-\ell$. 
\end{lemma} 

\begin{proof} 
Let $\phi\in D(A)$ be the image of $1\in A$ under an isomorphism $A(\ell)\to D(A)$ of graded right $A$-modules.  Define $\Phi:A*G(\ell)\to D(A*G)$ by $(\Phi(a*g))(c*k)=\phi (ag(c))\delta_{gk}$ where $\delta _g=\begin{cases} 1 & \textnormal { if } g=\e \\ 0 & \textnormal { if } g\neq \e\end{cases}$ ($\e\in G$ is the identity).  Since 
\begin{align*}
(\Phi((a*g)(b*h)))(c*k) & =(\Phi (ag(b)*gh))(c*k) \\
& = \phi(ag(b)(gh)(c))\delta_{(gh)k} \\
& = \phi (ag(b)g(h(c)))\delta_{ghk} \\
& = \phi (ag(bh(c)))\delta_{g(hk)} \\
& = (\Phi(a*g))(bh(c)*hk) \\
& = (\Phi(a*g))((b*h)(c*k)) \\
& =((\Phi(a*b))(b*h))(c*k),
\end{align*}
$\Phi$ is a graded right $A*G$-module homomorphism.  

For $\sum _ia_i*g_i\in A*G$, if $\Phi(\sum _ia_i*g_i)=0$, then 
$$\Phi\left(\sum _ia_i*g_i\right)(c*g_j^{-1})=\sum _i\phi(a_ig_i(c))\delta _{g_ig_j^{-1}}=\phi (a_jg_j(c))=(\phi a_j)(g_j(c))=0$$ for all $c*g_j^{-1}\in A*G$, so $\phi a_j=0\in D(A)$ for all $j$.  Since the map $A(\ell)\to D(A); \; a\to \phi a$ is an isomorphism of graded right $A$-modules, $a_j=0$ for all $j$, so $\Phi$ is injective.  Since $\dim _kD(A*G)=\dim _k(A*G(\ell))=(\dim _kA)|G|<\infty$, $\Phi$ is an isomorphism of graded right $A*G$-modules.  
\end{proof}  

\begin{proposition} \label{prop.n11} 
If $S$ is an AS-regular Koszul algebra over $k$ of dimension $d$ and $G\leq \GrAut S$ is a finite subgroup such that $\operatorname {char} k$ does not divide $|G|$, then $S*G$ is a generalized AS-regular Koszul algebra and $(S*G)^!\cong G*S^!$ is a Frobenius Koszul algebra of Gorenstein parameter $-d$. 
\end{proposition} 

\begin{proof} 
Since $S$ is generalized AS-regular by Lemma \ref{lem.gas}, $S*G$ is also generalized AS-regular by \cite [Lemma 13]{Ma2}. 
Since $S_0=k$ is a simple graded right $S*G$-module, 
$S_0$ has a graded projective resolution consisting of finitely generated graded projective right $S*G$-modules by \cite[Remark 3.16]{MM}, so 
$S*G$ is Koszul and $(S*G)^!\cong G*S^!$ by Proposition \ref{prop.49}.  By \cite [Proposition 5.10]{S}, $S^!$ is a Frobenius Koszul algebra of Gorenstein parameter $-d$, so $G*S^! \cong S^!*G^{o}$ is a graded Frobenius algebra of Gorenstein parameter $-d$ by Lemma \ref{lem.gF2}.  Since $G*S^!$ is finite dimensional (hence noetherian), $S^!_0$ has a graded projective resolution consisting of finitely generated graded projective right $G*S^!$-modules, so $G*S^!$ is Koszul by Proposition \ref{prop.49}. 
\end{proof} 

\subsection{Beilinson Algebras}

\begin{definition} Let $r\in \NN^+$.  
\begin{enumerate}
\item{} For a graded vector space $V=\bigoplus _{i\in \ZZ}V_i$, 
we define the graded vector space $V^{(r)} := \bigoplus _{i\in \ZZ}V_{ri}$.
\item{} For a graded algebra $A$, the $r$-th Veronese algebra of $A$ is the graded algebra 
$A^{(r)}$.  
\item{} For a graded algebra $A$, the $r$-th quasi-Veronese algebra of $A$ is the graded algebra defined by 
$$A^{[r]}:=
\begin{pmatrix} A^{(r)} & A(1)^{(r)} & \cdots & A(r-1)^{(r)} \\
A(-1)^{(r)} & A^{(r)} & \cdots & A(r-2)^{(r)} \\
\vdots & \vdots & \ddots & \vdots \\
A(-r+1)^{(r)} & A(-r+2)^{(r)} & \cdots & A^{(r)} \end{pmatrix}
$$
with the multiplication given by $(a_{ij})(b_{ij})=(\sum _{k=0}^{r-1}a_{kj}b_{ik})$ for $a_{ij}, b_{ij}\in A(j-i)^{(r)}$. 
\end{enumerate}
\end{definition} 

Let $A$ be a graded algebra and $r\in \NN^+$.  If a group $G$ acts on $A$, then $G$
naturally acts on $A^{[r]}$ by $g((a_{ij})):=(g(a_{ij}))$.

\begin{lemma} \label{lem.impo}
Let $A$ be a graded algebra, and $G\leq \GrAut A$ a 
subgroup. 
For $r\in \NN^+$, $A^{[r]}*G\cong (A*G)^{[r]}$ as graded algebras.
\end{lemma}

\begin{proof} 
Since the map 
$A(j)^{(r)}\otimes _kkG \to (A*G)(j)^{(r)}$
defined by $(a_i)*g\mapsto (a_i*g)$ where $a_i\in A_{ri+j}$ is an isomorphism of graded vector spaces for each $j\in \ZZ$, the map 
\begin{align*}
\phi : (A^{[r]})*G \to (A*G)^{[r]}
\end{align*}
defined by $\phi ((a_{ij})*g)=(a_{ij}*g)$ where $a_{ij}\in
A(j-i)^{(r)}$ is an isomorphism of graded vector spaces. 
It is easy to check that $\phi$ is an isomorphism of graded algebras.  
\end{proof}

\begin{definition}
\begin{enumerate}
\item
The Beilinson algebra of an AS-regular algebra $A$ over $A_0$ of Gorenstein parameter $\ell$ is defined by 
$$\nabla A:=(A^{[\ell]})_0=\begin{pmatrix} A_0 & A_1 & \cdots & A_{\ell-1} \\
0 & A_0 & \cdots & A_{\ell-2} \\
\vdots & \vdots & \ddots & \vdots \\
0 & 0 & \cdots & A_0\end{pmatrix}.$$ 

\item
The Beilinson algebra of a graded Frobenius algebra $A$ of Gorenstein parameter $-\ell$ is defined by 
$$\nabla A:=(A^{[\ell]})_0=\begin{pmatrix} A_0 & A_1 & \cdots & A_{\ell-1} \\
0 & A_0 & \cdots & A_{\ell-2} \\
\vdots & \vdots & \ddots & \vdots \\
0 & 0 & \cdots & A_0\end{pmatrix}.$$ 
\end{enumerate}
\end{definition} 

\begin{lemma} \textnormal{\cite[Lemma 3.13]{MU}}  
If $S$ is a noetherian AS-regular algebra over $k$ and $G\leq \GrAut S$ is a 
subgroup,
then $\nabla (S*G)\cong (\nabla S)*G$ as algebras. 
\end{lemma} 

\begin{lemma} \label{lem.10} 
If $A$ is a graded Frobenius algebra and $G\leq \GrAut A$ is a  
subgroup,  
then $\nabla (A*G)\cong (\nabla A)*G$ as algebras. 
\end{lemma} 

\begin{proof} If $A$ is a graded Frobenius algebra of Gorenstein parameter $-\ell$, then $A*G$ is a graded Frobenius algebra of Gorenstein parameter $-\ell$ by Lemma \ref{lem.gF2}, so 
$$\nabla (A*G)=((A*G)^{[\ell]})_0\cong (A^{[\ell]}*G)_0\cong (A^{[\ell]})_0*G=(\nabla A)*G$$
as algebras by Lemma \ref{lem.impo}.  
\end{proof}  

By the above lemmas, we can simply write $\nabla A*G$ for $(\nabla A)*G\cong \nabla (A*G)$.  

\begin{corollary} \label{cor.10}
If $S$ is an AS-regular Koszul algebra over $k$ of dimension $d$, and $G\leq \GrAut S$ is a finite subgroup such that $\operatorname {char} k$ does not divide $|G|$, then 
$$\nabla ((S*G)^!)\cong 
\nabla (S^!*G^o) \cong 
\nabla (S^!)*G^o \cong
G*\nabla (S^!)
$$ as algebras. 
\end{corollary}

\section{Derived Categories of Noncommutative Projective Schemes} 

Let $A$ be a graded right coherent algebra.  
Recall that the noncommutative projective scheme associated to $A$ is defined by the quotient category $\tails A:=\grmod A/\tors A$ where $\grmod A$ is the category of finitely presented graded right $A$-modules and $\tors A$ is the full subcategory of $\grmod A$ consisting of finite dimensional modules.  
In this section, we will find a finite dimensional algebra $\Lambda$ such that $\cD^b(\tails A)\cong \cD(\mod \Lambda)$ when $A$ is a ``noncommutative quotient isolated singularity".   

\subsection{BGG Correspondence}

For a graded algebra $A$, we denoted by $\grproj A$ the full subcategory of $\grmod A$ consisting of projective modules.  For a triangulated category $\cT$ and a set $T$ of objects in $\cT$, we denote by $\thick T$ the smallest full triangulated subcategory of $\cT$ containing $T$ and closed under isomorphism and direct summands. 
The proof of the following lemma is left to the reader.

\begin{lemma} \label{lem.TGP}
If $A$ is a graded right coherent algebra, then the following hold:  
\begin{enumerate}
\item{} $\cD^b(\grproj A)=\thick \{A(i)\}_{i\in \ZZ}$.   
\item{} If $A_0$ is finite dimensional and semisimple, then $\cD^b(\tors A)=\thick \{A_0(i)\}_{i\in \ZZ}$. 
\end{enumerate}
\end{lemma} 


The existence of the Koszul equivalence below is essential in this paper.  We refer to \cite {BGS} for the definitions of the functor $K$ and the categories $\cD^{\downarrow}(A), \cD^{\uparrow}(A^!)$.  

\begin{lemma} \textnormal{\cite [Theorem 2.12.1, Theorem 2.12.5]{BGS}} \label{lem.BGSK}
If $A$ is a Koszul algebra, then there exists an equivalence functor
$K:\cD^{\downarrow}(A)\to \cD^{\uparrow}(A^!)$ such that $K(X[i](j))=K(X)[i+j](-j)$, $K(A_0)=A^!$ and $K(D(A))\cong A^!_0$. 
\end{lemma} 

Note that the last isomorphism $K(D(A))\cong A^!_0$ follows from the fact $D(A)\cong \uHom_{A_0}(A, A_0)$ 
by Lemma \ref{lem.DH}.  

The next proposition, which is a generalization of the BGG correspondence, was essentially proved in \cite {MS}, however, our definition of Koszul algebra is slightly more general than that given in \cite {MS}, so we will repeat the proof to confirm the reader.

\begin{proposition}{\rm (cf. \cite [Proposition 4.1, Corollary 4.5]{MS})} \label{prop.MS} 
If $A$ is a Frobenius Koszul algebra of Gorenstein parameter $-\ell$ such that $A^!$ is graded right coherent, then there exists an equivalence 
$$K:\cD^b(\grmod A)\to \cD^b(\grmod A^!)$$
of triangulated categories, which induces an equivalence 
$$\overline K:\ugrmod A\to \cD^b(\tails A^!)$$
of triangulated categories.
\end{proposition}  

\begin{proof}  
Since $K(A_0(j))\cong A^![j](-j)\in \thick \{A^!(i)\}_{i\in \ZZ}$, and $K^{-1}(A^!(j))\cong A_0[j](-j)\in \thick \{A_0(i)\}_{i\in \ZZ}$ by Lemma \ref{lem.BGSK}, $K$ restricts to an equivalence $K:\thick \{A_0(i)\}_{i\in \ZZ}\to \thick \{A^!(i)\}_{i\in \ZZ}$.   
Since $A$ is finite dimensional, $\thick \{A_0(i)\}_{i\in \ZZ}=\cD^b(\tors A)=\cD^b(\grmod A)$ and since $A^!$ has finite global dimension, 
$\thick \{A^!(i)\}_{i\in \ZZ}=\cD^b(\grproj A^!)=\cD^b(\grmod A^!)$
by Lemma \ref{lem.TGP}, so $K$ induces an equivalence 
$$K:\cD^b(\grmod A)\to \cD^b(\grmod A^!)$$
of triangulated categories.  

Moreover,  
since $A$ is graded Frobenius of Gorenstein parameter $-\ell$, 
\begin{align*}
& K(A(j))\cong K(D(A)(j-\ell))\cong A^!_0[j-\ell](\ell-j)\in \thick \{A^!_0(i)\}_{i\in \ZZ}, \textnormal { and }\\ 
& K^{-1}(A^!_0(j))=D(A)[j](-j)\cong A[j](\ell-j)\in \thick \{A(i)\}_{i\in \ZZ},
\end{align*} 
so $K$ restricts to an equivalence functor $K:\thick \{A(i)\}_{i\in \ZZ}\to \thick \{A^!_0(i)\}_{i\in \ZZ}$.  
Since  
\begin{align*}
& \ugrmod A\cong \cD^b(\grmod A)/\cD^b(\grproj A)=\cD^b(\grmod A)/\thick \{A(i)\}_{i\in \ZZ}, \textnormal{ and } \\
& \cD^b(\tails A^!)=\cD^b(\grmod A^!/\tors A^!)\cong \cD^b(\grmod A^!)/\cD^b(\tors A^!)= \cD^b(\grmod A^!)/\thick \{A^!_0(i)\}_{i\in \ZZ}
\end{align*} 
by Lemma \ref{lem.TGP} and \cite [Theorem 4.4]{MS}, $K$ induces an equivalence 
$$\overline K:\ugrmod A\to \cD^b(\tails A^!)$$
of triangulated categories. 
\end{proof} 


Let $A$ be a Koszul algebra.  Since 
\begin{align*}
A_1\otimes _{A_0}\Hom_{A_0}(A_1, A_0)\cong \Hom_{A_0}(A_1, A_1\otimes _{A_0}A_0)\cong \Hom_{A_0}(A_1, A_1), \\
\Hom_{A_0^o}(A_1, A_0)\otimes _{A_0}A_1\cong \Hom_{A_0}(A_1, A_0\otimes _{A_0}A_1)\cong \Hom_{A_0^o}(A_1, A_1),
\end{align*} there exist elements
$v_{\a}\in A_1, {^*v_{\a}}\in \Hom_{A_0}(A_1, A_0), v_{\a}^*\in \Hom_{A_0^o}(A_1, A_0)$ such that $\sum v_{\a}\otimes {^*v_{\a}}$ corresponds to $\Id _{A_1}$ and $\sum v_{\a}^*\otimes v_{\a}$ corresponds to $\Id _{A_1}$ under the above isomorphisms, that is, for $w\in A_1$, 
\begin{align*}
& \sum v_{\a}({^*v_{\a}}(w))=w, \\
& \sum (v_{\a}^*(w))v_{\a}=w.
\end{align*}

We use the following lemma. 

\begin{lemma} \label{lem.cd} 
Let $A$ be a Koszul algebra.  For $V\in \Mod A_0^o$, the map 
$$\Hom_{A_0^o}(V, A_0)\otimes _{A_0}A^!\to \Hom_{A_0^o}(V, A_0\otimes _{A_0}A^!)\cong \Hom_{A_0^o}(V, A^!)$$
defined by $\Phi(\phi \otimes a)(v)=\phi(v)a$ is an isomorphism in $\GrMod A^!$.  
\end{lemma}  

\begin{proposition} \label{prop.em} 
Let $A$ be a Koszul algebra.  If $M\in \lin A^!$, then $K(DE_{A^!}(M))$ is a linear resolution of $M$.  
\end{proposition} 

\begin{proof} Write $N:=E_{A^!}(M)\in \GrMod A^o$.  By Lemma \ref{lem.DH}, we identify $D(N_i)$ with $\Hom_{A_0^o}(N_i, A_0)$ as $A_0$-modules.  By the proof of \cite [Theorem 2.12.1]{BGS}, $K(DN)$ is a complex 
$$K(DN)^i=(DN)_{-i}\otimes _{A_0}A^!(-i)=D(E_{A^!}(M)_{i})\otimes _{A_0}A^!(-i)=D\uExt^i_{A^!}(M, A_0)\otimes _{A_0}A^!(-i)$$ with differentials 
$d:D(N_{i})\otimes _{A_0}A^!(-i)\to  D(N_{i-1})\otimes _{A_0}A^!(-i+1)$ given by 
$$d(\phi \otimes a)=(-1)^i\sum \phi v_{\a}\otimes {^*v_{\a}}a.$$
By Lemma \ref{lem.cd}, we have a commutative diagram 
$$\begin{CD}
\Hom_{A_0^o}(N_{i}, A_0)\otimes _{A_0}A^!(-i) @>d>> \Hom_{A_0^o}(N_{i-1}, A_0)\otimes _{A_0}A^!(-i+1) \\ 
@V{\Phi}V\cong V @V{\Phi}V\cong V \\
\Hom_{A_0^o}(N_i, A^!)(-i) @>d'>> \Hom_{A_0^o}(N_{i-1}, A^!)(-i+1).
\end{CD}$$
For $v\in N_{i-1}$, 
\begin{align*}
d'(\Phi (\phi\otimes a))(v) & = \Phi(d(\phi\otimes a))(v) =\Phi \left((-1)^i\sum \phi v_{\a}\otimes {^*v_{\a}}a\right )(v) \\
& = (-1)^i\sum (\phi v_{\a})(v){^*v_{\a}a} 
= (-1)^i\sum v_{\a}^*\phi (v_{\a}v)a \\
& = (-1)^i\sum v_{\a}^*\Phi (\phi \otimes a)(v_{\a}v)
\end{align*} 
so $d'(f)(v)=(-1)^i\sum v_{\a}^*f(v_{\a}v)$ for $f\in \Hom_{A_0^o}(N_i, A^!)$, hence $K(DE_{A^!}(M))$ is a linear resolution of $M$ by \cite [Proposition 2.14.5]{BGS} and the comment after the proposition (see also the remark after \cite [Definition 2.8.1]{BGS}).  
\end{proof}




\begin{lemma} \label{lem.em2}   
If $A$ is a Koszul algebra and $M\in \lin A^!$, then $DK^{-1}(M)\cong E_{A^!}(M)$ in $\cD^b(\grmod A^o)$.  
\end{lemma}

\begin{proof} By Proposition \ref{prop.em}, $K(DE_{A^!}(M))\cong M$ in $\cD^b(\grmod A^!)$.  Since $M\in \cD^{\uparrow}(A^!)$ and $E_{A^!}(M)$ is locally finite,  
$DK^{-1}(M)\cong DDE_{A^!}(M)\cong E_{A^!}(M)$ in $\cD^b(\grmod A^o)$. 
\end{proof} 

\begin{lemma} \label{lem.em3}
If $A$ is a Koszul algebra and $M\in \lin A$, then $E_A(\Omega ^iM(i))\cong E_A(M)_{\geq i}(i)$ in $\GrMod (A^!)^o$ for all $i\in \NN$.  
\end{lemma} 

\begin{proof} 
By Proposition \ref{prop.em}, we have an exact sequence 
$$0\to \Omega ^1M\to D\uHom_A(M, A_0)\otimes _{A_0}A\to M\to 0$$
in $\GrMod A$, which induces an exact sequence 
$$0\to \uHom_A(M, A_0)\to \uHom_A(D\uHom_A(M, A_0)\otimes _{A_0}A, A_0)\to \uHom_A(\Omega ^1M, A_0)\to \uExt^1_A(M, A_0)\to 0$$
 and isomorphisms $\uExt^i_A(\Omega ^1M, A_0)\cong \uExt ^{i+1}_A(M, A_0)$ for $i\geq 1$.   Since 
\begin{align*}
\uHom_A(D\uHom_A(M, A_0)\otimes _{A_0}A, A_0) 
& \cong \uHom_{A_0}(D\uHom_A(M, A_0), \uHom_A(A, A_0)) \\
& \cong \uHom_{A_0}(D\uHom_A(M, A_0), A_0) \\
& \cong DD\uHom_A(M, A_0) \\
& \cong \uHom_A(M, A_0)
\end{align*}
in $\GrMod (A^!)^o$ by Lemma \ref{lem.DH}, $\uHom_A(\Omega _1M, A_0)\cong \uExt^1_A(M, A_0)$.  It follows that 
\begin{align*}
E_A(\Omega ^1M(1))
& :=\bigoplus _{i=0}^{\infty}\uExt^i_A(\Omega ^1M(1), A_0)
 =\bigoplus _{i=0}^{\infty}\uExt^{i+1}_A(M, A_0) \\
& \cong \bigoplus _{i=1}^{\infty}\uExt^i_A(M, A_0)(1) 
 \cong E_A(M)_{\geq 1}(1)
\end{align*}
in $\GrMod (A^!)^o$.  Since $\Omega ^1M(1)\in \lin A$, the result follows by induction. 
\end{proof} 

\subsection{Tilting Objects} 

\begin{definition} Let $\cT$ be a triangulated category.  An object $T\in \cT$ is called tilting if 
\begin{enumerate}
\item{} $\thick (T)=\cT$, and 
\item{} $\Hom_{\cT}(T, T[i])=0$ for all $i\neq 0$. 
\end{enumerate}
\end{definition} 

A tilting object plays an essential role in this paper due to the following result.

\begin{theorem} \textnormal{(cf. \cite [Theorem 2.2]{IT})} \label{thm.tt}
Let $\cT$ be an algebraic Krull-Schmidt triangulated category and $T\in \cT$ a tilting object.  If $\gldim \End _{\cT}(T)<\infty$, then $\cT\cong \cD^b(\mod \End _{\cT}(T))$.
\end{theorem}

\begin{lemma} \label{lem.21} 
If $A$ is a graded Frobenius algebra of Gorenstein parameter $-\ell$, then $D(A_{\geq i}(i)) \cong A(\ell-i)/A(\ell-i)_{\geq 1}$ in $\grmod A$.
\end{lemma} 

\begin{proof} Since $A$ is a graded Frobenius algebra, $D(A)\cong A(\ell)$ in $\grmod A$.  
For each $i\in \NN$, the inclusion $A_{\geq i}\to A$ in $\grmod A^o$ induces a surjection $D(A)\to D(A_{\geq i})$ in $\grmod A$.  Since 
$$D(A_{\geq i})_j=\Hom_k((A_{\geq i})_{-j}, k)=\begin{cases} \Hom_k(A_{-j}, k)=D(A)_j & \textnormal{ if } -j\geq i \\
0 & \textnormal { if } -j<i, \end{cases}$$
$$D(A_{\geq i})\cong D(A)/D(A)_{\geq 1-i}\cong A(\ell)/A(\ell)_{\geq 1-i}=A/A_{\geq \ell+1-i}(\ell)$$
in $\grmod A$, so
$$
D(A_{\geq i}(i))=D(A_{\geq i})(-i)\cong A/A_{\geq \ell+1-i}(\ell)(-i) = A(\ell-i)/A(\ell-i)_{\geq 1}$$
in $\grmod A$. 
\end{proof} 

\begin{lemma} \label{lem.Y} 
If $A$ is a graded Frobenius algebra of Gorenstein parameter $-\ell$ such that $\gldim A_0<\infty$, then
\begin{enumerate} 
\item{} $\bigoplus _{i=0}^{\ell-1}A(i)/A(i)_{\geq 1}$ is a tilting object for $\ugrmod A$, 
\item{} $\End _{\ugrmod A}\left(\bigoplus _{i=0}^{\ell-1}A(i)/A(i)_{\geq 1}\right)\cong \nabla A$, and 
\item{} there is an equivalence $\ugrmod A\cong \cD^b(\mod \nabla A)$ of triangulated categories. 
\end{enumerate}
\end{lemma} 

\begin{proof} Since 
$A(i)/A(i)_{\geq 1}=A(i)$ are graded projective right $A$-modules for all $i\geq \ell$ and $A(i)/A(i)_{\geq 1}=0$ for all $i<0$, $U:=\bigoplus _{i=0}^{\ell-1}A(i)/A(i)_{\geq 1}$ is a tilting object for $\ugrmod A$ by \cite [Theorem 3.3 (2)]{Y}. 

If $P$ is an indecomposable graded projective right $A$-module, then $P$ is a direct summand of $A(-i)$ for some $i$, so $P=eA(-i)$ for some idempotent $e\in A$, hence $P_i\neq 0$.   Since $D(P)$ is a direct summand of $D(A(-i))\cong A(\ell+i)$ as a graded left $A$-module, 
$D(P)$ is an indecomposable graded projective left $A$-module, so 
$D(P_{\ell+i})\cong D(P)_{-\ell-i}\neq 0$ by the same argument, hence $P_{\ell+i}\neq 0$.  
Since $U_i=0$ for all $i>0$ and $i\leq - \ell$, there is no projective direct summand in $U$, so $\End _{\ugrmod A}(U)\cong \nabla A$ by the proof of \cite [Lemma 3.9]{Y}, and  there is an equivalence $\ugrmod A\cong \cD^b(\mod \nabla A)$ of triangulated categories by \cite [Theorem 3.11 (2)]{Y}. 
\end{proof} 

\begin{remark}
The equivalence $\ugrmod A\cong \cD^b(\mod \nabla A)$ of triangulated categories also follows from \cite[Theorem 4.22.3]{MM}.
\end{remark} 


Let $A$ be a graded right coherent algebra.  Recall that $\pi :\grmod A\to \tails A$ denotes the quotient functor. 

\begin{theorem} \label{thm.22}
If $A$ is a Frobenius Koszul algebra of Gorenstein parameter $-\ell$ such that $A^!$ is graded right coherent AS-regular over $A^!_0$,
then 
\begin{enumerate}
\item{} $\bigoplus _{i=1}^{\ell}\pi \Omega ^iA^!_0(i)$ is a tilting object for $\cD^b(\tails A^!)$,
\item{}  $\End_{\cA^!}\left(\bigoplus _{i=1}^{\ell}\pi \Omega ^iA^!_0(i)\right)\cong \nabla A$, and 
\item{} there is an equivalence $\cD^b(\tails A^!)\cong \cD^b(\mod \nabla A)$ of triangulated categories.    
\end{enumerate} 
\end{theorem} 

\begin{proof} 
Since $\Omega^iA^!_0(i)\in \lin A^!$, 
\begin{align*}
DK^{-1}\left(\bigoplus _{i=1}^{\ell}\Omega ^iA^!_0(i)\right) 
& \cong E_{A^!}\left(\bigoplus _{i=1}^{\ell}\Omega ^iA^!_0(i)\right) 
  \cong \bigoplus _{i=1}^{\ell}E_{A^!}(\Omega ^iA^!_0(i)) \\
& \cong \bigoplus _{i=1}^{\ell}E_{A^!}(A^!_0)_{\geq i}(i) \cong \bigoplus _{i=1}^{\ell}A_{\geq i}(i) 
\end{align*}
in $\cD^b(\grmod A^o)$ by Lemma \ref{lem.em2} and Lemma \ref{lem.em3}, so  
$$K^{-1}\left(\bigoplus _{i=1}^{\ell}\Omega ^iA^!_0(i)\right)
\cong D\left(\bigoplus _{i=1}^{\ell}A_{\geq i}(i)\right)
\cong \bigoplus _{i=1}^{\ell}A(\ell-i)/A(\ell-i)_{\geq 1}\cong \bigoplus _{i=0}^{\ell-1}A(i)/A(i)_{\geq 1}$$
in $\grmod A$ by Lemma \ref{lem.21}.  
Since  
the functor $K$ induces an equivalence $\overline K:\ugrmod A\to \cD^b(\tails A^!)$ of triangulated categories by Proposition \ref{prop.MS},
$\bigoplus _{i=1}^{\ell}\pi \Omega ^iA^!_0(i)=\overline K(\bigoplus _{i=0}^{\ell-1}A(i)/A(i)_{\geq 1})$ is a tilting object for $\cD^b(\tails A^!)$ such that  
$$
\End_{\cA^!}\left(\bigoplus _{i=1}^{\ell}\pi \Omega ^iA^!_0(i)\right)
\cong \End_{\ugrmod A}\left(\bigoplus _{i=0}^{\ell-1}A(i)/A(i)_{\geq 1}\right) \\
\cong \nabla A$$
and there is an equivalence $\cD^b(\tails A^!)\cong \ugrmod A\cong \cD^b(\mod \nabla A)$ of triangulated categories by Lemma \ref{lem.Y}.
\end{proof} 

\begin{remark} \label{rem.sg} 
If $A$ is a noetherian AS-regular algebra over $A_0$ of dimension $d$, then the noncommutative projective scheme $Y:=\Proj _{nc}A$ associated to $A$ is regarded as a quantum projective space of dimension $d-1$ over $\Spec _{nc}A_0$.  If $A$ is Koszul, then the $i$-th sheaf of differentials on $Y$ is defined by $\Omega_Y^i=\pi \Omega _A^{i+1}A_0\in \tails A$ for $i\in \NN$ in \cite [Definition 5.7]{Mbc}.  
By the above theorem, $\bigoplus _{i=0}^{d-1}\Omega _Y^i(i)=(\bigoplus _{i=1}^{d}\pi \Omega ^i_
{A}A_0(i))(-1)$ is a tilting object for $\cD^b(\tails A)$.  In fact, via the equivalence $\overline K:\ugrmod A^!\to \cD^b(\tails A)$, the tilting object $\bigoplus _{i=0}^{d-1}\Omega _Y^i(i)$ for $\cD^b(\tails A)$ corresponds to the tilting object $\bigoplus _{i=1}^{d}A^!_{\geq i}(i)$ for $\ugrmod A^!$.   
\end{remark} 

\begin{corollary} \label{cor.tsg}
If $S$ is a noetherian AS-regular Koszul algebra over $k$ of dimension $d$ and $G\leq \GrAut S$ is a finite subgroup such that 
$\operatorname{char} k$ does not divide $|G|$, 
then 
\begin{enumerate}
\item{} $\bigoplus _{i=1}^{d}\pi \Omega ^i_{S*G}kG(i)$ is a tilting object for $\cD^b(\tails S*G)$,
\item{} $\End_{\cS*G}\left(\bigoplus _{i=1}^{d}\pi \Omega ^i_{S*G}kG(i)\right)\cong  G*\nabla(S^!)$, and 
\item{} there is an equivalence $\cD^b(\tails S*G)\cong \cD^b(\mod  G*\nabla(S^!))$ of triangulated categories.    
\end{enumerate}
\end{corollary} 

\begin{proof} Since $S*G$ is a noetherian AS-regular Koszul algebra over $(S*G)_0=kG$ of dimension $d$, and $(S*G)^!$ is a Frobenius Koszul algebra of Gorenstein parameter $-d$ by 
Lemma \ref{lem5} and Proposition \ref{prop.n11},  $\bigoplus _{i=1}^{d}\pi \Omega ^i_{S*G}kG(i)$ is a tilting object for $\cD^b(\tails S*G)$ such that 
$$\End_{\cS*G}\left(\bigoplus _{i=1}^{d}\pi \Omega ^i_{S*G}kG(i)\right)
\cong \nabla ((S*G)^!)
\cong G*\nabla (S^!),$$
and there is an equivalence $\cD^b(\tails S*G)\cong \cD^b(\mod G*\nabla (S^!))$ of triangulated categories by Theorem \ref{thm.22}
and Corollary \ref{cor.10}. 
\end{proof} 


\begin{lemma} \label{lem.ksg}
If $S$ is an AS-regular Koszul algebra over $k$ of dimension $d$, $G\leq \GrAut S$ is a finite subgroup such that 
$\operatorname{char} k$ does not divide $|G|$, and $e=\frac{1}{|G|}\sum _{g\in G}g\in kG\subset S*G$ so that $e(S*G)e\cong S^G$, then $(\Omega _{S*G}^ikG)e\cong f_*\Omega _S^ik$ in $\GrMod S^G$ for every $i\in \NN^+$ where $f:S^G\to S$ is the inclusion. 
\end{lemma} 

\begin{proof} Since $S*G$ is Koszul by Proposition \ref{prop.n11}, the result follows from Theorem \ref{thm.cor}. 
\end{proof} 

Let $A$ be a graded right coherent algebra.  Recall our notations that $\cM=\pi M\in \tails A$ is the image of $M\in \grmod A$, and $\Hom_{\cA}(\cM, \cN):=\Hom_{\tails A}(\pi M, \pi N)$.   

\begin{lemma} \label{lem:amptor}
Let $A$ be a right noetherian graded algebra and $e \in A$ an idempotent such
that $eAe$ is a right noetherian graded algebra. If $Ae \in \grmod eAe$, and
\[(-)e : \tails A \to \tails eAe \]
is an equivalence functor, then $A/(e)$ is an torsion $A$-module.
\end{lemma}

\begin{proof}
Let $F: \grmod A/(e) \to \grmod A$ be the restriction functor induced by the natural epimorphism from $A$ to $A/(e)$.
It is easy to check that $F$ is faithful.
Let $G:=(-)e:\grmod A \to \grmod eAe$.
Since $F$ and $G$ preserve torsion modules, these functors induce the functors
\[ \tails A/(e) \overset{\overline{F}}{\longrightarrow} \tails A \overset{\overline{G}=(-)e}{\longrightarrow} \tails eAe. \]
Since $F$ is faithful,
\begin{align*}
\Hom_{\cA/(e)}(\cM, \cN)
&\cong \lim_{n \to \infty}\Hom_{A/(e)}(M_{\geq n}, N)\\
&\hookrightarrow \lim_{n \to \infty}\Hom_{A}(M_{\geq n}, N) \\
&\cong \Hom_{\cA}(\cM, \cN) \cong \Hom_{\cA}(\overline{F}(\cM), \overline{F}(\cN)),
\end{align*}
so $\overline{F}$ is also faithful. Moreover, for any $\cM \in \tails A/(e)$, it is easy to see that $\overline{G}\overline{F}(\cM)=0$. Since $\overline{G}=(-)e$ is an equivalence functor, $\overline{F}(\cM) = 0$.
It follows from the faithfulness of $\overline{F}$ that
$ \Hom_{\cA/(e)}(\cM, \cM) \hookrightarrow \Hom_{\cA}(\overline{F}(\cM), \overline{F}(\cM)) = 0,$
so the identity morphism for $\cM$ is zero. Thus $\cM = 0$ for any $\cM \in \tails A/(e)$.
We see that $A/(e) \in \tors A/(e)$ and hence $A/(e) \in \tors A$.
\end{proof}

 
\begin{proposition} \label{prop.amp}
Let $A$ be a right noetherian connected graded algebra, $G \leq \GrAut A$ a finite subgroup
such that $\fchar k$ does not divide $|G|$, and $e = \frac{1}{|G|} \sum _{g\in G}g \in kG\subset A*G$.
Then the following are equivalent:
\begin{enumerate}
\item $A*G/(e)$ is finite dimensional over $k$.
\item $(-)e:\tails A*G\to \tails A^G$ is an equivalence functor.
\end{enumerate}
\end{proposition}

\begin{proof}
This follows from \cite[Theorem 2.13]{MU} and Lemma \ref{lem:amptor}.
\end{proof}

\begin{remark}
The above equivalent conditions are closely related to the condition that $S^G$ is a graded isolated singularity \cite [Theorem 3.10]{MU}.
The condition (2) is called ``ample group action'' in \cite{MU}.
\end{remark}  

\begin{theorem} \label{thm.ksg}
Let $S$ be a noetherian AS-regular Koszul algebra over $k$ of dimension $d$, $G\leq \GrAut S$ a finite subgroup such that 
$\operatorname{char} k$ does not divide $|G|$, and $f:S^G\to S$ the inclusion map.  If $S*G/(e)$ is finite dimensional over $k$,    
then 
\begin{enumerate}
\item{} $\bigoplus _{i=1}^{d}\pi f_*\Omega ^i_Sk(i)$ is a tilting object for $\cD^b(\tails S^G)$, 
\item{} $\End _{\cS^G}\left(\bigoplus _{i=1}^{d}\pi f_*\Omega ^i_Sk(i)\right)\cong G*\nabla (S^!)$, and
\item{} there is an equivalence $\cD^b(\tails S^G)\cong \cD^b(\mod G*\nabla (S^!))$
of triangulated categories. 
\end{enumerate}
\end{theorem}

\begin{proof} Since $S*G/(e)$ is finite dimensional over $k$, $(-)e:\tails S*G\to \tails S^G$ is an equivalence functor by Proposition \ref{prop.amp}.  Since $\bigoplus _{i=1}^{d}\pi \Omega ^i_{S*G}kG(i)$ is a tilting object for $\cD^b(\tails S*G)$ by Corollary \ref{cor.tsg}, and $(\bigoplus _{i=1}^{d}\pi \Omega ^i_{S*G}kG(i))e\cong \bigoplus_{i=1}^{d}\pi f_*\Omega ^i_Sk(i)$ in $\tails S^G$ by Lemma \ref{lem.ksg}, $\bigoplus _{i=1}^{d}\pi f_*\Omega ^i_{S}k(i)$ is a tilting object for $\cD^b(\tails S^G)$ such that 
$$\End _{\cS^G}\left(\bigoplus _{i=1}^{d}\pi f_*\Omega ^i_Sk(i)\right)
\cong \End _{\cS*G}\left(\bigoplus_{i=1}^{d}\pi \Omega ^i_{S*G}kG(i)\right)
\cong G*\nabla (S^!),$$
and there is an equivalence $\cD^b(\tails S^G)\cong \cD^b(\tails S*G)\cong \cD^b(\mod G*\nabla (S^!))$
of triangulated categories by Corollary \ref{cor.tsg}.
\end{proof} 

\begin{remark} 
In the setting of the above theorem,  the inclusion map $f:S^G\to S$ induces a functor $f_*:\tails S\to \tails S^G$.   Since $S*G/(e)$ is finite dimensional over $k$, the functor $(-)e:\GrMod S*G\to \GrMod S^G$ induces an equivalence $(-)e:\tails S*G\to \tails S^G$ by Proposition \ref{prop.amp}.  Write $X=\Proj _{nc}S$ and $Y=\Proj _{nc}S*G$.  
It is easy to see that $\Omega ^{d-1}_X(d-1), \dots, \Omega^1_X(1), \Omega^0_X$ is a full strong exceptional sequence for $\cD^b(\tails S)$ (cf. \cite [Proposition 4.8, Theorem 5.11]{Mbc}).  Since $\End_{\cS^G}(f_*\Omega ^i_X(i))\cong kG$, $f_*\Omega ^{d-1}_X(d-1), \dots, f_*\Omega^1_X(1), f_*\Omega^0_X$ is no longer an exceptional sequence for $\cD^b(\tails S^G)$ in the usual sense unless $G$ is trivial, however,  
it is a full strong exceptional sequence ``over $kG$".   
In fact, $\bigoplus _{i=0}^{d-1}\Omega _Y^i(i)$ is a tilting object for $\cD^b(\tails S*G)$ by Remark \ref{rem.sg}, so 
$$\bigoplus _{i=0}^{d-1}f_*\Omega _X^i(i)=\left(\bigoplus _{i=1}^{d}\pi f_*\Omega ^i_Sk(i)\right)(-1)
\cong \left(\bigoplus _{i=1}^{d}\pi \Omega ^i_{S*G}kG(i)\right)e(-1)
=\left(\bigoplus _{i=0}^{d-1}\Omega _Y^i(i)\right)e$$ 
is a tilting object for $\cD^b(\tails S^G)$ as in the above theorem.  If $kG$ is a finite direct product of $k$, then we may obtain a full strong exceptional sequence for $\cD^b(\tails S^G)$ by replacing each $f_*\Omega _X^i(i)$ with its set of indecomposable direct summands in a suitable order.  If $kG$ is not a finite direct product of $k$, then we may still form a full strong exceptional sequence by deleting isomorphic indecomposable direct summands (cf. \cite[Corollary 2.12]{IT}).  We thank the referee for pointing this out.    
\end{remark}

\section{Stable Categories of Maximal Cohen-Macaulay Modules}

Let $A$ be a 
noetherian AS-Gorenstein algebra over $k$.
Then the graded singularity category is defined by the Verdier localization $\cD_{\rm Sg}^{\rm gr}(A) := \cD^b(\grmod A)/ \cD^b(\grproj A)$.
We denote the localization functor by $\upsilon: \cD^b(\grmod A)\to \cD_{\rm Sg}^{\rm gr}(A)$.  By \cite {B}, $\cD_{\rm Sg}^{\rm gr}(A)\cong \uCM^{\ZZ}A$ the stable category of graded maximal Cohen-Macaulay modules over $A$.  
In this section, we will find a finite dimensional algebra $\Gamma$ such that $\uCM^{\ZZ}A\cong \cD(\mod \Gamma)$ when $A$ is a ``noncommutative quotient isolated singularity".   

\subsection{Tilting Objects}

Let $A$ be a 
noetherian AS-Gorenstein algebra over $k$ of Gorenstein parameter $\ell$.  If $\ell>0$, then there exists the embedding
$\Phi:=\Phi_0:\cD_{\rm Sg}^{\rm gr}(A)\to \cD^b(\tails A)$ by Orlov \cite{O}.
Unfortunately, $\Phi \upsilon \neq \pi $, but we have the following result.   

\begin{lemma} \label{lem.phph} 
Let $A$ be a 
noetherian AS-Gorenstein algebra over $k$ of positive Gorenstein parameter, and $M\in \grmod A$,
If $M_{\geq 0}=M$ and $\Hom_A(M, A(i))=0$ for all $i\leq 0$, then $\Phi (\upsilon M)\cong \pi M$.  
\end{lemma} 

\begin{proof} This follows from the proof of \cite [Theorem 4.3]{A}.  
\end{proof} 

Let $A$ be a right noetherian graded algebra. For $M \in \grmod A$, we define
\[ \depth M = \inf\{ i | \uExt_A^i(A_0,M) \neq 0 \}. \]
If $A$ is a right noetherian graded algebra such that
$A_0$ is a finite dimensional semisimple algebra (eg. $A$ is Koszul) and $M \in \grmod A$,
then one can show that
\[ \uExt_A^i(T,M) = 0 \]
for any $T \in \tors A$ and any $i < \depth M$.

\begin{lemma} \label{lem.dep}
Let $A$ be a right noetherian graded algebra such that
$A_0$ is a finite dimensional semisimple algebra and
$M, N \in\grmod A$. If $\depth N \geq  2$, then the natural map
$$\uHom_A(M,N) \to  \uHom_\cA(\cM,\cN)$$
is an isomorphism of vector spaces.
\end{lemma}

\begin{proof}
The proof is same as that of \cite[Lemma 2.9]{Mmc}.
\end{proof}

\begin{setting} \label{set} We fix the following setting:
\begin{enumerate}
\item{} $S$ is a 
noetherian AS-regular Koszul algebra over $k$ of dimension $d\geq 2$ 
(so that the Gorenstein parameter is $\ell=d$). 
\item{} $G\leq \GrAut S$ is 
a finite subgroup such that $\operatorname{char} k$ does not divide $|G|$.
\item{} $e=\frac{1}{|G|}\sum _{g\in G}g\in kG\subset S*G$ and $e'=1-e\in kG\subset S*G$.  
\item{} $S^G$ is AS-Gorenstein and $S*G/(e)$ is finite dimensional over $k$ (so that $S^G$ is a ``Gorenstein isolated singularity").  
\item{} $U= \bigoplus _{i=1}^{d}\Omega ^i_{S*G}kG(i)$. (Note that $U$ is a graded $kG$-$S*G$ bimodule by Remark \ref{rem.left0}.)
\item{} $(K^{\bullet}, d)$ is a linear resolution of $kG$ over $S$.
\end{enumerate}
\end{setting}

Let $f:S^G\to S$ be the inclusion map.  In Setting \ref{set}, $\pi Ue\cong \pi (\bigoplus _{i=1}^{d}f_*\Omega ^i_Sk(i)) \in \tails S^G$ is a tilting object for $\cD^b(\tails S^G)$ such that $\End _{\cS^G}(\pi Ue)\cong G*\nabla (S^!)$ by Lemma \ref{lem.ksg} and Theorem \ref{thm.ksg}. 

\begin{lemma} \label{lem.0}
In Setting \ref{set}, $\Phi (\upsilon e'Ue)=\pi e'Ue$.  
\end{lemma} 

\begin{proof}  
Since $S*G$ is Koszul, $\Omega ^i_{S*G}kG$ is generated in degree $i$, so $\Omega ^i_{S*G}kG(i)$ is generated in degree 0 for every $1\leq i\leq d$, 
hence there exists a surjective homomorphism $\bigoplus e'(S*G)\to e'U\to 0$.  It follows that $(e'Ue)_{\geq 0}=e'Ue\in \grmod S^G$, so  it is enough to show that 
$\Hom_{S^G}(e'Ue, S^G(i))=0$ for all $i\leq 0$ by Lemma \ref{lem.phph}. 
Since $S^G$ is an AS-Gorenstein algebra of dimension $d\geq 2$, 
it follows that $\depth _{S^G}(S^G)\geq 2$.  Since $S*G$ is an AS-regular algebra over $kG$ of dimension $d\geq 2$ by Lemma \ref{lem5}, $\depth _{S*G}S=\depth _{S*G}e(S*G)=\depth _{S*G}(S*G)\geq 2$.   By Lemma \ref{lem.dep},  
\begin{align*}
\Hom_{S^G}(e'Ue, S^G(i)) & \cong \Hom_{\cS^G}(\pi e'Ue, \pi S^G(i)) \\
& \cong \Hom_{\cS*G}(\pi e'U, \pi S(i)) \\
& \cong \Hom_{S*G}(e'U, e(S*G)(i)) \\
& \hookrightarrow  \Hom _{S*G}\left(\bigoplus e'(S*G), e(S*G)(i)\right) \\
& \cong \bigoplus (e(S*G)e')_i.
\end{align*}  
Since $(e(S*G)e')_0=ekGe'=0$, $(e(S*G)e')_i=0$ for all $i\leq 0$, hence the result.  
\end{proof} 

\begin{lemma} \label{lem.1} 
In Setting \ref{set}, $e'(S*G)e(i)\in \operatorname{thick} (e'Ue)$ in $\cD^b(\grmod S^G)$ for all $i\in \ZZ$. 
\end{lemma} 

\begin{proof} Recall that $e'Ue=\bigoplus _{i=1}^de'(\Omega ^i_{S*G}kG)e(i)$ so that $e'(\Omega ^i_{S*G}kG)e(i)\in \operatorname{thick} (e'Ue)$ for $1\leq i\leq d$.
The (truncated) linear resolution of $kG$ over $S*G$  
$$0\to \Omega_{S*G}^{i}kG\to K^{i-1}\to K^{i-2}\to \cdots \to K^1\to S*G\to kG\to 0$$
induce a long exact sequence 
$$0\to e'(\Omega ^i_{S*G}kG)e(i)\to e'K^{i-1}e(i)\to e'K^{i-2}e(i)\to \cdots \to e'K^1e(i)\to e'(S*G)e(i)\to e'kGe(i)=0$$
for $1\leq i\leq d$. 
It follows that $e'(S*G)e(1)\cong e'K^0e(1)\cong e'(\Omega ^1_{S*G}kG)e(1)\in \operatorname{thick} (e'Ue)$. Since 
$K^i$ is a graded projective right $S*G$-module generated in degree $i$, $K^i
\in \add (S*G(-i))$.  It follows that $e'K^1e(2)\in \add (e'(S*G)e(1))$, so $e'(S*G)e(2)\cong e'(\Omega ^1_{S*G}kG)e(2)\in \operatorname{thick} (e'Ue)$. 
By induction, $e'(S*G)e(i)\in \operatorname{thick} (e'Ue)$ for $1\leq i\leq d$. 

Since $K^d=S*G(-d)$, we have a long exact sequence 
$$0\to e'(S*G)e(i)\to e'K^{d-1}e(i+d)\to \cdots \to e'K^1e(i+d)\to e'(S*G)e(i+d)\to 0$$
for each $i\in \ZZ$.  Since $e'K^ie(j+i)\in \add (e'(S*G)e(j))\subset \thick (e'Ue)$ for all $0\leq i\leq d, 1\leq j\leq d$, we can show that $e'(S*G)e(i)\in \operatorname{thick} (e'Ue)$ for all $i\in \ZZ$ by induction.  
\end{proof} 

\begin{lemma} \label{lem.2} 
In Setting \ref{set}, $\upsilon (S*G)e(i)\in \operatorname{thick} (\upsilon e'Ue)$ in $\cD_{\rm Sg}^{\rm gr}(S^G)$ for all $i\in \ZZ$. 
\end{lemma} 

\begin{proof} Since $e(S*G)e(i)\cong S^G(i)$ in $\grmod S^G$,
$\upsilon (S*G)e(i)= \upsilon e(S*G)e(i)\oplus \upsilon e'(S*G)e(i)= \upsilon e'(S*G)e(i)$ in $\cD_{\rm Sg}^{\rm gr}(S^G)$.
Since $\upsilon :\cD^b(\grmod S^G)\to \cD_{\rm Sg}^{\rm gr}(S^G)$ is a localization functor,
$$\upsilon (S*G)e(i)=\upsilon e'(S*G)e(i)\in  \operatorname{thick} (\upsilon e'Ue)$$
for all $i\in \ZZ$ by Lemma \ref{lem.1}.  
\end{proof} 

\begin{lemma} \label{lem.3} 
In Setting \ref{set}, $\thick \{\upsilon (S*G)e(i)\mid i\in \ZZ\}=\cD_{\rm Sg}^{\rm gr}(S^G)$.  
\end{lemma}

\begin{proof}
For any object $X \in \cD_{\rm Sg}^{\rm gr}(S^G)$, there exists $M \in \CM^{\ZZ}(S^G)$
such that $X \cong \upsilon M$ by Buchweitz \cite{B}.
Since  
$S$ is a $(d-1)$-cluster tilting object in $\CM^{\ZZ}(S^G)$ by \cite[Theorem 3.15]{MU},
we have an exact sequence 
$$0\to T_{d-1}\to \cdots \to T_1\to T_0\to M\to 0$$
in $\grmod S^G$ where $T_i\in \add\{S(i)\mid i\in \ZZ\}$ 
as in the proof of \cite[Theorem 5.10]{U},
so it follows that $M\in \thick \{S(i)\mid i\in \ZZ\}$ in $\cD(\grmod S^G)$.
Hence 
$$X \cong \upsilon M\in \thick \{\upsilon S(i)\mid i\in \ZZ\} = \thick \{\upsilon (S*G)e(i)\mid i\in \ZZ\}$$
in $\cD_{\rm Sg}^{\rm gr}(S^G)$. 
\end{proof}

\begin{proposition} \label{prop.2} 
In Setting \ref{set}, 
$\upsilon e'Ue$ is a tilting object for $\cD_{\rm Sg}^{\rm gr}(S^G)$. 
\end{proposition} 

\begin{proof} 
Since $\pi e'Ue$ is a direct summand of a tilting object $\pi Ue$ for $\cD^b(\tails S^G)$, 
$$\Hom_{\cD_{\rm Sg}^{\rm gr}(S^G)}(\upsilon e'Ue, \upsilon e'Ue[i])\cong \Hom_{\cS^G}(\pi e'Ue, \pi e'Ue[i])\subset \Hom_{\cS^G}(\pi Ue, \pi Ue[i])=0$$ for all $i\neq 0$ by Lemma \ref{lem.0}.  Since $\upsilon (S*G)e(i)\in \operatorname {thick} (\upsilon e'Ue)$ for all $i\in \ZZ$ by Lemma \ref{lem.2}, we have that $\operatorname {thick} (\upsilon e'Ue)=\thick\{\upsilon (S*G)e(i)\mid i\in \ZZ\}= \cD_{\rm Sg}^{\rm gr}(S^G)$ by Lemma \ref{lem.3}, hence the result.  
\end{proof}

\begin{lemma} \label{lem.e'UeMCM}
In Setting \ref{set},
$e'Ue \in \CM^{\ZZ}(S^G)$.
\end{lemma}

\begin{proof}
The (truncated) linear resolution of $kG$ over $S*G$ induce a long exact sequence 
$$0\to e'(\Omega_{S*G}^{i}kG)e\to e'K^{i-1}e\to e'K^{i-2}e\to \cdots \to e'K^1e\to e'(S*G)e\to e'kGe=0$$
for $1\leq i\leq d$.
It follows that
\begin{align} \label{eq.e'UeMCM}
e' (\Omega_{S*G}^{1}kG)e \cong e'(S*G)e, \; \textrm{and} \;
0 \to e'(\Omega_{S*G}^{i}kG)e \to  e'K^{i-1}e \to e' (\Omega_{S*G}^{i-1}kG)e \to 0
\end{align}
for $2\leq i\leq d$.
Since $e'(S*G)e, e'K^{i-1}e \in \add\{ (S*G)e(i)|i \in \ZZ\} = \add\{ S(i)|i \in \ZZ\}$,
these are graded maximal Cohen-Macaulay over $S^G$,
so it follows from (\ref{eq.e'UeMCM}) that $e'(\Omega_{S*G}^{i}kG)e \in \CM^{\ZZ}(S^G)$ for each $1\leq i\leq d$.
Hence we have $e'Ue =\bigoplus _{i=1}^de'(\Omega_{S*G}^{i}kG)e(i) \in \CM^{\ZZ}(S^G)$.
\end{proof}

\begin{theorem} \label{thm.CMtilt} 
In Setting \ref{set}, 
$e'Ue$ is a tilting object for $\uCM^{\ZZ}(S^G)$.
\end{theorem} 

\begin{proof}
Under the equivalence functor $\uCM^{\ZZ}(S^G) \to \cD_{\rm Sg}^{\rm gr}(S^G)$ given by Buchweitz \cite{B},
$e'Ue$ corresponds to $\upsilon e'Ue$ by Lemma \ref{lem.e'UeMCM}.
Thus Proposition \ref{prop.2} implies the result.
\end{proof}

\subsection{Endomorphism Algebras} 
In this subsection, we calculate the endomorphism algebra of the tilting object found in the previous subsection. 

\begin{lemma} \label{lem.ee} Let $\cC$ be an abelian category, and $M\in \cC$ an object.  For an idempotent element $e\in \End _{\cC}(M)$,    
$\End _{\cC}(eM)\cong e\End _{\cC}(M)e$ as rings where $eM:=\Im e$.
\end{lemma}

Let $A$ be a Koszul algebra over $A_0$ and $e\in A$ an idempotent.  Since $e^2=e$, it follows that $e\in A_0$.  
As stated in Remark \ref{rem.left0}, the linear resolution $(K^{\bullet}, d^{\bullet})$ of $A_0$ 
is a complex of graded $A_0$-$A$ bimodules.  Since $K^i$ is a graded projective right $A$-module generated in degree $i$, $eK^i$ is also a graded projective right $A$-module generated in degree $i$.  Since the functor $e(-)=eA_0\otimes _{A_0}-:\GrMod A_0^o\otimes A\to \GrMod A$ is exact,  
$(eK^{\bullet}, ed^{\bullet})$ is a linear resolution of $eA_0$, so $e\Omega ^iA_0:=e\Im d^i\cong \Im ed^i=:\Omega ^i(eA_0)$.
We define the map $e_i:\Omega ^iA_0(i)\to \Omega ^iA_0(i)$ induced by the left multiplication by $e$.  

\begin{lemma} \label{lem.eee} 
Let $A$ be an AS-regular Koszul algebra over $A_0$ of dimension $d$, and $U:=\bigoplus _{i=1}^d\Omega ^iA_0(i)$.  For an idempotent $e\in A$, $\End_{\cA}(\pi eU)\cong \pi \bar e\End _{\cA}(\pi U)\pi \bar e$ as rings where $\bar e=\bigoplus _{i=1}^de_i\in \End _A(U)$. 
\end{lemma} 

\begin{proof}  Since  
$\Im e_i=e\Omega ^iA_0(i)$,  
$$\Im \bar e\cong \bigoplus _{i=1}^d\Im e_i\cong \bigoplus _{i=1}^de\Omega ^iA_0(i)=:eU.$$   
Since $\pi: \grmod A\to \tails A$ is an exact functor, 
$\Im (\pi \bar e)\cong \pi \Im \bar e\cong \pi eU$.
Since $e_i\in \End _A(\Omega^iA_0(i))$ is an idempotent, $\pi \bar e\in \End_{\cA}(\pi eU)$ is an idempotent, so $\End_{\cA}(\pi eU)\cong \pi \bar e\End _{\cA}(\pi U)\pi \bar e$ as rings by Lemma \ref{lem.ee}. 
\end{proof}

\begin{lemma} \label{lem.gF} 
Let $A$ be a graded Frobenius algebra of Gorenstein parameter $-\ell$ and $a\in A_i$.  If $ba=0$ for every $b\in A_{\ell-i}$, then $a=0$.  
\end{lemma}    

\begin{proof} Since $A$ is a graded Frobenius algebra of Gorenstein parameter $-\ell$, there exists an isomorphism $\Phi:A\to (DA)(-\ell)$ of graded right $A$-modules.  If $\phi :=\Phi (1)\in (DA)(-\ell)_0=(DA)_{-\ell}$, then the map $\<-, -\>:A\times A\to k$ defined by $\<x, y\>=\phi (xy)$ is a nondegenerate (associative) bilinear form. 
For $x=\sum _{j=0}^{\ell}x_j\in A$, if $j=\ell-i$, then $\<x_j,a\>=\phi (x_ja)=\phi (0)=0$ by assumption and, if $j\neq \ell-i$, then $x_ja\in A_{i+j}\neq A_{\ell}$, so $\<x_j, a\>=\phi (x_ja)=0$ since $\phi \in (DA)_{-\ell}$, hence $\<x, a\>=\<\sum _{j=0}^{\ell}x_j, a\>=\sum _{j=0}^{\ell}\<x_j, a\>=0$.   Since $\<-, -\>$ is nondegenerate, $a=0$.  
\end{proof}  

\begin{lemma} \label{lem.pre2g}
Let $A$ be a graded Frobenius algebra of Gorenstein parameter $-\ell$.  For $0\leq i, j \leq \ell$, the map $A_{j-i}\to \Hom_{A^o}(A_{\geq i}(i),A_{\geq j}(j))$ defined by $a\mapsto \cdot a $ is an isomorphism. 
\end{lemma} 

\begin{proof} The exact sequences
$$0\to A_{\geq i}\to A\to A/A_{\geq i}\to 0, \quad 0\to A_{\geq j}\to A\to A/A_{\geq j}\to 0$$
in $\GrMod A$ induce exact sequences
\begin{align*}
& 0\to \Hom_{A^o}(A/A_{\geq i}(i), A(j))\to \Hom_{A^o}(A(i), A(j))\to \Hom_{A^o}(A_{\geq i}(i), A(j))\to \Ext^1_{A^o}(A/A_{\geq i}(i), A(j)), \\
& 0\to \Hom_{A^o}(A_{\geq i}(i), A_{\geq j}(j))\to \Hom_{A^o}(A_{\geq i}(i), A(j))\to \Hom_{A^o}(A_{\geq i}(i), A/A_{\geq j}(j)).
\end{align*}
Since $A$ is graded Frobenius, $\Ext^1_{A^o}(A/A_{\geq i}(i), A(j))=0$.
Let $\phi\in\Hom_{A^o}(A/A_{\geq i}(i), A(j))$. 
Since $\ell-j\geq 0$, for every $a\in A_{\ell-(j-i)}$, $a\phi (\overline 1)=\phi (a \overline 1)=\phi (0)=0$.  
Since $A$ is graded Frobenius and $\phi (\overline 1)\in A(j)_{-i} = A_{j-i}$, it follows that $\phi (\overline 1)=0$ by Lemma \ref{lem.gF}, so $\phi=0$, hence $\Hom_{A^o}(A/A_{\geq i}(i), A(j))=0$.
Since $\Hom_{A^o}(A(i), A(j))\to \Hom_{A^o}(A_{\geq i}(i), A(j))$ is the restriction of the domain, the map $A_{j-i} \cong \Hom_A(A(i), A(j))\to \Hom_A(A_{\geq i}(i), A(j))$ defined by $a\mapsto \cdot a $ is an isomorphism.
Clearly, $\Hom_{A^o}(A_{\geq i}(i), A/A_{\geq j}(j))=0$. Since $\Hom_{A^o}(A_{\geq i}(i), A_{\geq j}(j))\to \Hom_{A^o}(A_{\geq i}(i), A(j))$ is the extension of the codomain, the map 
$$\Hom_{A^o}(A_{\geq i}(i), A_{\geq j}(j))\to \Hom_{A^o}(A_{\geq i}(i), A(j))$$
defined by $\cdot a \mapsto \cdot a$ is an isomorphism, hence the result.  
\end{proof}

If $A$ is an AS-regular Koszul algebra over $A_0$ of dimension $d$ such that $A^!$ is graded Frobenius,
then there exists an isomorphism
$$\Psi:\End_{\cal A}\left(\bigoplus _{i=1}^d\pi \Omega ^iA_0(i)\right)\to \nabla (A^!):=
\begin{pmatrix} A^!_0 & A^!_1 & \cdots & A^!_{d-1} \\ 0 & A^!_0 & \cdots & A^!_{d-2} \\ \vdots & \vdots & \ddots &  \vdots  \\ 0 & 0 & \cdots & A^!_0 \end{pmatrix}$$
by Theorem \ref{thm.22}, so it induces an isomorphism 
$$\Psi_i:\End_{\cA}(\pi \Omega ^iA_0(i))\to A^!_0 \cong A_0$$ 
for each $1\leq i\leq d$. 

\begin{lemma} \label{lem.ei} 
Let $A$ be an AS-regular Koszul algebra over $A_0$
such that $A^!$ is graded Frobenius, and let $e\in A$ be an idempotent.  For each $i\in \NN^+$, the isomorphism  
$\Psi_i:\End_{\cA}(\pi \Omega ^iA_0(i))\to A_0$ sends $\pi e_i$ to $e$.  
\end{lemma} 

\begin{proof}  
%
For every $1\leq i\leq d$, we have $e\Omega ^iA_0(i)\cong \Omega^i(eA_0)(i)\in \lin A$, so
\[
\xymatrix{
 E_A(\Omega ^iA_0(i)) \ar[r]^{E_A(e\cdot )} \ar[d]^{\cong} &E_A(\Omega ^iA_0(i)) \ar[d]^{\cong}\\
 E_A(A_0)_{\geq i}(i) \ar[r]^{E_A(e\cdot )} &E_A(A_0)_{\geq i}(i)}
\]
commutes. Moreover, since ${A^!}$ is graded Frobenius, we have the composition of isomorphisms
\begin{align*}
&\End_{\cA}(\pi \Omega ^iA_0(i))                              &&\pi e_i =\pi e \cdot \\
&\cong \End_{\ugrmod A^!}(\overline{K}^{-1}(\Omega ^iA_0(j))) &&\mapsto \overline{K}^{-1}(e\cdot) && (\textrm{by Proposition \ref{prop.MS}}) \\
&\cong \End_{\ugrmod A^!}(DE_A(\Omega ^iA_0(i)))              &&\mapsto DE_A(e \cdot) && (\textrm{by Lemma \ref{lem.em2}}) \\
&\cong \End_{\ugrmod {A^!}^o}(E_A(\Omega ^iA_0(i)))           &&\mapsto E_A(e \cdot)\\
&\cong \End_{\ugrmod {A^!}^o}(E_A(A_0)_{\geq i}(i))           &&\mapsto E_A(e \cdot) && (\textrm{by Lemma \ref{lem.em3}}) \\
&\cong \End_{\grmod {A^!}^o}(E_A(A_0)_{\geq i}(i))            &&\mapsto E_A(e \cdot)\\
&\cong \End_{\grmod {A^!}^o}({A^!}_{\geq i}(i))               &&\mapsto E_A(e \cdot)= \cdot (e \cdot)\\
&\cong {A^!}_{0}                                              &&\mapsto e \cdot && (\textrm{by Lemma \ref{lem.pre2g}}) \\
&\cong A_0,                                                   &&\mapsto e.
\end{align*}
We can check that this gives $\Psi_i$, so the result follows.
\end{proof}

 We now return to Setting \ref{set}.

\begin{proposition} \label{prop.0} 
In Setting \ref{set}, 
$\End _{\cD_{\rm Sg}^{\rm gr}(S^G)}(\upsilon e'Ue)\cong \tilde{e'}(G*\nabla (S^!))\tilde{e'}$ where 
$$\tilde{e'}=\begin{pmatrix} e' & 0 & \cdots & 0 \\ 0 & e' & \cdots & 0 \\ \vdots & \vdots & \ddots &  \vdots  \\ 0 & 0 & \cdots & e' \end{pmatrix}\in \begin{pmatrix} kG & * & \cdots & * \\ 0 & kG & \cdots & * \\ \vdots & \vdots & \ddots &  \vdots  \\ 0 & 0 & \cdots & kG \end{pmatrix}=G*\nabla (S^!).$$ 
\end{proposition} 

\begin{proof} By Corollary \ref{cor.tsg}, there exists an isomorphism 
$$\Psi:\End_{\cS*G}(\pi U)\cong \nabla ((S*G)^!)= \begin{pmatrix} 
kG & * & \cdots & * \\
0 & kG & \cdots & * \\
\vdots & \vdots & \ddots & \vdots \\
0 & 0 & \cdots & kG \end{pmatrix}.$$ 
By Lemma \ref{lem.ei}, 
$$\Psi(\pi \bar{e'})=\Psi\left(\bigoplus \pi e'_i\right)=\begin{pmatrix} \Psi_1(\pi e'_1) & 0 & \cdots & 0 \\ 0 & \Psi_2(\pi e'_2) & \cdots & 0 \\ \vdots & \vdots & \ddots &  \vdots  \\ 0 & 0 & \cdots & \Psi_d(\pi e'_d) \end{pmatrix}=\begin{pmatrix} e' & 0 & \cdots & 0 \\ 0 & e' & \cdots & 0 \\ \vdots & \vdots & \ddots &  \vdots  \\ 0 & 0 & \cdots & e' \end{pmatrix}=\tilde{e'}.$$
Moreover one can check that $\tilde{e'}$ is an invariant under the isomorphisms 
$$\nabla ((S*G)^!)
\cong G*\nabla (S^!), $$
so we get 
$$\begin{array}{lll}
\End _{\cD_{\rm Sg}^{\rm gr}(S^G)}(\upsilon e'Ue)
& \cong \End _{\cS^G}(\pi e'Ue) & \textnormal {(by Lemma \ref{lem.0})} \\
& \cong \End _{\cS*G}(\pi e'U) & \textnormal {(by Proposition \ref{prop.amp})} \\
& \cong \pi \bar{e'}\End_{\cS*G}(\pi U)\pi \bar{e'} & \textnormal {(by Lemma \ref{lem.eee})} \\
& \cong \tilde{e'}(\nabla (S^!*G))\tilde{e'}\\
& \cong \tilde{e'}(G*\nabla (S^!))\tilde{e'}.
\end{array}$$  
\end{proof}

\begin{theorem} \label{thm.main} 
In Setting \ref{set}, 
$$\uCM^{\ZZ}(S^G)\cong \cD^b(\mod \tilde{e'}(G*\nabla (S^!))\tilde{e'})$$
as triangulated categories where 
$$\tilde {e'}=\begin{pmatrix} e' & 0 & \cdots & 0 \\ 0 & e' & \cdots & 0 \\ \vdots & \vdots & \ddots &  \vdots  \\ 0 & 0 & \cdots & e' \end{pmatrix}\in \begin{pmatrix} kG & * & \cdots & * \\ 0 & kG & \cdots & * \\ \vdots & \vdots & \ddots &  \vdots  \\ 0 & 0 & \cdots & kG \end{pmatrix}=G*\nabla (S^!).$$ 
\end{theorem}

\begin{proof} 
By Proposition \ref{prop.0},
$$\End_{\uCM^{\ZZ}(S^G)}(e'Ue) \cong \End_{\cD_{\rm Sg}^{\rm gr}(S^G)}(\upsilon e'Ue) \cong \tilde{e'}(G*\nabla (S^!))\tilde{e'}.
$$
Since $kG$ is semisimple, $e'kGe'$ has finite global dimension by \cite[Lemma 4.5]{KX},
so
$$ \tilde{e'}(G*\nabla (S^!))\tilde{e'} =
\begin{pmatrix} e'kGe' & * & \cdots & * \\ 0 & e'kGe' & \cdots & * \\ \vdots & \vdots & \ddots &  \vdots  \\ 0 & 0 & \cdots & e'kGe' \end{pmatrix}$$
also has finite global dimension.
Since $\uCM^{\ZZ}(S^G)$ is an algebraic Krull-Schmidt triangulated category,
the result follows from Theorem \ref{thm.tt} and Theorem \ref{thm.CMtilt}.
\end{proof}

\section{Examples}

The aim of this section is to provide an explicit example of Theorem \ref{thm.main}.
For a connected graded algebra $A$ and $G \leq \GrAut A$, to check whether $A*G/(e)$ is finite dimensional over $k$ or not,
we will use a quiver presentation of $A*G/(e)$.

For the rest of this paper, $k$ denotes an algebraically closed field of characteristic $0$.
Let $A= k\<x_1,\dots,x_n\>/(f_1,\dots, f_h)$ be a noetherian connected graded algebra.
Let $G$ be a cyclic group generated by $g= \diag (\xi^{a_1},\dots,\xi^{a_n})$
with a primitive $r$-th root of unity $\xi$ and integers $a_j$ satisfying 
$0 < a_j \leq r$ for any $j$, and $|G|=r$.
We assume that $G$ acts on $A$ by $g\cdot x_j = \xi^{a_j} x_j$.
In this case, $G$ acts on $k\<x_1,\dots,x_n\>$ by the same action.

The McKay quiver $Q_{G}$ of the cyclic group $G$ is given as follows.
\begin{itemize}
\item The set of vertices is $\ZZ/r\ZZ$. 
\item The set of arrows is $\{ x_{j,i} : i- a_j \to i \mid i \in \ZZ/r\ZZ,\; 1\leq j\leq n \}$. 
\end{itemize}

\begin{remark}
We will often write $x_j$ instead of $x_{j,i}$ when there is no possibility of confusion.
\end{remark}

For any $i \in \ZZ/r\ZZ$, we define $\rho_i$ by $\frac{1}{r}\sum _{p=0}^{r-1} \xi^{ip} g^p \in kG$.
Notice that $g \rho_i = \xi^{-i}\rho_i$.

\begin{prop}
Define a map $\phi: k\<x_1,\dots,x_n\> *G \to kQ_{G}$ by 
\[ \phi(1* \rho_i) = e_i \quad \textrm{and} \quad 
\phi (x_{s_1}x_{s_2}\cdots x_{s_m}*\rho_{i}) = x_{s_1, i-a_{s_m}\cdots -a_{s_2}}\cdots x_{s_{m-1}, i-a_{s_{m}}}x_{s_m, i}.\]
Then $\phi$ is an algebra isomorphism.
\end{prop}

\begin{proof}
Since
\begin{align*}
&(x_{s_1}\cdots x_{s_{m-1}}x_{s_m}*\rho_i) (x_{t_1}\cdots x_{t_{l-1}}x_{t_l}*\rho_{i'})\\
&= (x_{s_1}\cdots x_{s_{m-1}}x_{s_m}*\frac{1}{r}\sum_{p=0}^{r-1}\xi^{ip}g^p)(x_{t_1}\cdots x_{t_{l-1}}x_{t_l}*\rho_{i'})\\
&= \frac{1}{r}\sum_{p=0}^{r-1}\xi^{ip} x_{s_1}\cdots x_{s_{m-1}}x_{s_m}g^p(x_{t_1}\cdots x_{t_{l-1}}x_{t_l})*g^p\rho_{i'}\\
&= \frac{1}{r}\sum_{p=0}^{r-1}\xi^{ip} x_{s_1}\cdots x_{s_{m-1}}x_{s_m}\xi^{p(a_{t_1}+\cdots+a_{t_l})}x_{t_1}\cdots x_{t_{l-1}}x_{t_l}*\xi^{-pi'}\rho_{i'}\\
&= \frac{1}{r}\sum_{p=0}^{r-1}\xi^{p(i-i'+a_{t_1}+\cdots+a_{t_l})} x_{s_1}\cdots x_{s_{m-1}}x_{s_m}x_{t_1}\cdots x_{t_{l-1}}x_{t_l}*\rho_{i'}\\
&=\begin{cases}
x_{s_1}\dots x_{s_{m-1}}x_{s_m}x_{t_1}\cdots x_{t_{l-1}}x_{t_l}*\rho_{i'} &\textrm{if}\; i=i'-a_{t_1}\cdots -a_{t_l} \;\textrm{in}\; \ZZ/r\ZZ \\
0 &\textrm{if}\; i \neq i'-a_{t_1}\cdots -a_{t_l} \;\textrm{in}\; \ZZ/r\ZZ, \\
\end{cases}
\end{align*}
it follows that $\phi$ is an algebra homomorphism.
It is easy to check that $\phi$ is bijective.
\end{proof}

As an immediate consequence of the above proposition, we have the following theorem.

\begin{theorem} \label{thm:S*G}
A quiver presentation of $A*G$ is given by the McKay quiver $Q_{G}$ with relations
\[ \phi(f_j * \rho_i) = 0 \qquad (1\leq j \leq h, \; i \in \ZZ/r\ZZ). \]
\end{theorem}

The quiver presentation of $A*G$ obtained above will be denoted by the pair $(Q_{A*G},R_{A*G})$
where $Q_{A*G}$ is the quiver $Q_{G}$, and $R_{A*G}$ is the set of relations.

\begin{corollary} \label{cor:S*G/(e)}
Let $e:=\frac{1}{|G|}\sum _{g \in G}g\in kG\subset A*G$
be the idempotent.
Then a quiver presentation of $A*G/(e)$ is obtained from that of $A*G$ by removing the vertex $0$ and all arrows incident to it.
\end{corollary}

\begin{proof}
Since $e = \frac{1}{r}\sum _{i=0}^{r-1}g^i= \rho_0$, this is clear.
\end{proof}

The quiver presentation of $A*G/(e)$ obtained above will be denoted by the pair $(Q_{A*G/(e)},R_{A*G/(e)})$
where $Q_{A*G/(e)}$ is the quiver obtained by removing the vertex $0$ from $Q_{G}$,
and $R_{A*G/(e)}$ is the set of relations.

Using Theorem \ref{thm:S*G} and Corollary \ref{cor:S*G/(e)}, we now present a nontrivial example.
We consider the case when the acting group $G$ is a subgroup of $\HSL(S)$ but not a subgroup of $\SL(d,k)$.

\begin{example}
Let $S$ be $k\<x_1,x_2,x_3,x_4 \>$ having six defining relations
$$x_1^2+x_2^2,\;\; x_1x_3+x_3x_1,\;\; x_1x_4+x_4x_1,\;\; x_2x_3+x_3x_2,\;\; x_2x_4+x_4x_2,\;\; x_3x_4+x_4x_3,$$
with $\deg x_1=\deg x_2=\deg x_3=\deg x_4=1$.
Then $S$ is a noetherian AS-regular Koszul algebra over $k$ of dimension $4$.
Let $G$ be a cyclic group generated by $g = \diag(1,-1,-1,-1)$.
Then $g$ defines a graded algebra automorphism of $S$, so $G$ naturally acts on $S$.
Clearly $|G|=2$.
Moreover one can check that the homological determinant of $g$ is equal to 1 (although $\det g\neq 1$), so it follows that $S^G$ is AS-Gorenstein of dimension $4$.

By Theorem \ref{thm:S*G}, the quiver presentation $(Q_{S*G},R_{S*G})$ is given by
\vspace{2truemm}

\noindent
\begin{minipage}{.2\textwidth}
\qquad
\end{minipage}
\begin{minipage}{.3\textwidth}
\begin{align*}
\xymatrix@C=1pc@R=4pc{
0 \ar@<1ex>[d]|(0.3){x_2} \ar@<2ex>[d]|(0.5){x_3} \ar@<3ex>[d]|(0.7){x_4} \ar@(ul,ur)[]|{x_1} \\
1 \ar@<1ex>[u]|(0.3){x_2} \ar@<2ex>[u]|(0.5){x_3} \ar@<3ex>[u]|(0.7){x_4} \ar@(dr,dl)[]|{x_1}
}
\end{align*}
\end{minipage}
\begin{minipage}{.5\textwidth}
$x_1^2=-x_2^2,\;\; x_1x_3=-x_3x_1,$\\
$x_1x_4=-x_4x_1,\;\; x_2x_3=-x_3x_2,$\\
$x_2x_4=-x_4x_2,\;\; x_3x_4=-x_4x_3.$
\end{minipage}
\vspace{2truemm}

\noindent
Furthermore, by Corollary \ref{cor:S*G/(e)}, the quiver presentation $(Q_{S*G/(e)},R_{S*G/(e)})$ is given by
\vspace{2truemm}

\noindent
\begin{minipage}{.2\textwidth}
\qquad
\end{minipage}
\begin{minipage}{.3\textwidth}
\begin{align*}
\xymatrix@C=1pc@R=4pc{
1  \ar@(dr,dl)[]|{x_1}
}
\end{align*}
\end{minipage}
\begin{minipage}{.5\textwidth}
$x_1^2= 0,$
\end{minipage}
\vspace{2truemm}

\noindent
so we see that $S*G/(e)$ is finite dimensional over $k$.

The Koszul dual $S^!$ is $k\<x_1,x_2,x_3,x_4 \>$ having ten defining relations
\begin{align*}
&x_2^2-x_1^2, \;\; x_3x_1-x_1x_3,\;\; x_4x_1-x_1x_4,\;\; x_2x_1,\;\; x_1x_2,\\
&x_3x_2-x_2x_3,\;\; x_4x_2-x_2x_4,\;\; x_4x_3-x_3x_4
,\;\; x_3^2,\;\; x_4^2,
\end{align*}
with $\deg x_1=\deg x_2=\deg x_3=\deg x_4=1$.
Then a quiver presentation of the Beilinson algebra $\nabla (S^!)$ is given as follows.
\vspace{2truemm}

\noindent
\begin{minipage}{.5\textwidth}
\begin{align*}
\xymatrix@C=3pc@R=1pc{
0 \ar@<2.4ex>[r]|(0.28){x_1} \ar@<0.8ex>[r]|(0.40){x_2} \ar@<-0.8ex>[r]|(0.52){x_3} \ar@<-2.4ex>[r]|(0.64){x_4} 
&1 \ar@<2.4ex>[r]|(0.28){x_1} \ar@<0.8ex>[r]|(0.40){x_2} \ar@<-0.8ex>[r]|(0.52){x_3} \ar@<-2.4ex>[r]|(0.64){x_4} 
&2 \ar@<2.4ex>[r]|(0.28){x_1} \ar@<0.8ex>[r]|(0.40){x_2} \ar@<-0.8ex>[r]|(0.52){x_3} \ar@<-2.4ex>[r]|(0.64){x_4} 
&3
}
\end{align*}
\end{minipage}
\begin{minipage}{.5\textwidth} 
$x_1^2=x_2^2, \; x_3x_1=x_1x_3, \; x_4x_1=x_1x_4,$\\
$x_3x_2=x_2x_3,\; x_4x_2=x_2x_4, \;x_4x_3=x_3x_4,$\\
$x_2x_1 = x_1x_2 = x_3^2 = x_4^2=0.$
\end{minipage}
\vspace{2truemm}

\noindent
By \cite[Section 2.3]{RR}, it follows that a quiver presentation of the skew group algebra $G*\nabla (S^!)$ is
\vspace{2truemm}

\noindent
\begin{minipage}{.6\textwidth}
\begin{align*}
\xymatrix@C=3pc@R=4pc{
(0,0) \ar[r]|{x_1} \ar@<0.8ex>[rd]|(0.2){x_2} \ar[rd]|(0.3){x_3} \ar@<-0.8ex>[rd]|(0.4){x_4} 
&(1,0) \ar[r]|{x_1} \ar@<0.8ex>[rd]|(0.2){x_2} \ar[rd]|(0.3){x_3} \ar@<-0.8ex>[rd]|(0.4){x_4}  
&(2,0) \ar[r]|{x_1} \ar@<0.8ex>[rd]|(0.2){x_2} \ar[rd]|(0.3){x_3} \ar@<-0.8ex>[rd]|(0.4){x_4} 
&(3,0) \\
(0,1) \ar[r]|{x_1} \ar@<0.8ex>[ru]|(0.2){x_2} \ar[ru]|(0.3){x_3} \ar@<-0.8ex>[ru]|(0.4){x_4} 
&(1,1) \ar[r]|{x_1} \ar@<0.8ex>[ru]|(0.2){x_2} \ar[ru]|(0.3){x_3} \ar@<-0.8ex>[ru]|(0.4){x_4} 
&(2,1) \ar[r]|{x_1} \ar@<0.8ex>[ru]|(0.2){x_2} \ar[ru]|(0.3){x_3} \ar@<-0.8ex>[ru]|(0.4){x_4} 
&(3,1)
}
\end{align*}
\end{minipage}
\begin{minipage}{.4\textwidth}
\begin{align*}
&x_1^2=x_2^2, \; x_3x_1=x_1x_3,\\
&x_1x_4=x_4x_1,\;x_3x_2=x_2x_3,\\
&x_2x_4=x_4x_2,\;x_4x_3=x_3x_4,\\
&x_2x_1 = x_1x_2 = x_3^2 = x_4^2=0.
\end{align*}
\end{minipage}
\vspace{2truemm}

\noindent
Since $\tilde{e'} = (0,1)+(1,1)+(2,1)+(3,1)$ in the present setting, a quiver presentation of $\tilde{e'} (G*\nabla (S^!)) \tilde{e'}$ is obtained by
\vspace{2truemm}

\noindent
\begin{minipage}{.6\textwidth}
\begin{align*}
\xymatrix@C=3pc@R=3pc{
(0,1) \ar[r]|{x_1} \ar@<1.6ex>@/^1.5pc/[rr]|(0.3){x_2x_3} \ar@<0.8ex>@/^1.5pc/[rr]|(0.5){x_2x_4} \ar@/^1.5pc/[rr]|(0.7){x_3x_4} 
&(1,1) \ar[r]|{x_1}  \ar@/_1.5pc/[rr]|(0.3){x_2x_3} \ar@<-0.8ex>@/_1.5pc/[rr]|(0.5){x_2x_4} \ar@<-1.6ex>@/_1.5pc/[rr]|(0.7){x_3x_4} 
&(2,1) \ar[r]|{x_1} 
&(3,1)
}
\end{align*}
\end{minipage}
\begin{minipage}{.4\textwidth}
\begin{align} \label{eq.qr}
\begin{split}
&x_2x_3x_1 = x_1x_2x_3 =0\\
&x_2x_4x_1 = x_1x_2x_4 =0\\
&x_3x_4x_1 = x_1x_3x_4\\
&x_1^3=0
\end{split}
\end{align}
\end{minipage}
\vspace{2truemm}

\noindent
Hence, if we denote by $(Q,R)$ the quiver with relations in (\ref{eq.qr}), then we have a triangle equivalence
$$\uCM^{\ZZ}(S^G)\cong \cD^b(\mod kQ/(R))$$
by Theorem \ref{thm.main}.
\end{example}

\end{document}